\documentclass[a4paper,11pt]{amsart}
\usepackage[utf8x]{inputenc}

\usepackage{amssymb,amsfonts}
\usepackage{amsmath,amsthm}
\usepackage[all,arc, cmtip]{xy}
\usepackage{enumerate}
\usepackage{paralist}
\usepackage{booktabs} 
\usepackage{array} 
\usepackage{mathrsfs}
\usepackage{url}
\usepackage{color}
\usepackage{todonotes} 
\usepackage{hyperref} 
\usepackage{mathtools}
\usepackage[mathscr]{euscript}

\usepackage{tikz-cd}

\newcommand{\mc}{\mathcal}
\newcommand{\mb}{\mathbb}
\newcommand{\ms}{\mathscr}

\def\cd{{\centerdot}}
\def\defeq{{\vcentcolon=}}

\def\proj{{\mathcal{P}}}

\DeclareMathOperator\ke{ker}
\DeclareMathOperator\im{im}

\newcommand{\binr}[2]{\ar@<.5ex>[r]^{#1} \ar@<-.5ex>[r]_{#2}}
\newcommand{\bind}[2]{\ar@<.5ex>[d]^{#1} \ar@<-.5ex>[d]_{#2}}
\newdir{ >}{{}*!/-5pt/\dir{>}}
\newdir^{ (}{{}*!/-5pt/@^{(}}

\def\id{{\rm id}}

\def\bF{{\mathbb{F}}}   
\def\bG{{\mathbb{G}}}   
\def\bQ{{\mathbb{Q}}}   
\def\bZ{{\mathbb{Z}}}   
\def\bN{{\mathbb{N}}}   

\def\cM{{\mathcal M}}
\def\cO{{\mathcal O}}

\newcommand{\longto}{\longrightarrow}

\DeclareMathOperator\End{End}

\DeclareMathOperator\GL{GL}
\DeclareMathOperator\Hom{Hom}
\DeclareMathOperator\Mat{Mat}
\DeclareMathOperator\Mod{Mod}
\DeclareMathOperator\Pol{Pol}
\DeclareMathOperator\Sym{Sym}

\DeclareMathOperator\cHom{\mathcal{H}om}

\DeclareMathOperator\uEnd{\underline{End}}
\DeclareMathOperator\uHom{\underline{Hom}}
\DeclareMathOperator\uSpec{\underline{Spec}}

\DeclareMathOperator\uP{\underline{\proj}}

\newtheorem{thm}{Theorem}[section]
\newtheorem*{thm*}{Theorem}
\newtheorem{thmdef}[thm]{Theorem / Definition}
\newtheorem{cor}[thm]{Corollary}
\newtheorem{prop}[thm]{Proposition}
\newtheorem{lem}[thm]{Lemma}

\theoremstyle{definition}

\newtheorem{defn}[thm]{Definition}
\newtheorem*{defn*}{Definition}
\newtheorem{ex}[thm]{Example}
\newtheorem{mainex}[thm]{Main Example}

\newtheorem{rem}[thm]{Remark}

\theoremstyle{remark}

\makeatletter
\let\c@equation\c@thm
\makeatother
\numberwithin{equation}{section}
\title[Exterior power operations via binary complexes]
{Exterior power operations on higher $K$-groups via binary complexes}

\author{Tom Harris}
\address{University Printing House\\ Shaftesbury Avenue\\ Cambridge CB2 8BS\\ United Kingdom}
\email{tharris@cambridge.org}

\author{Bernhard K\"{o}ck}
\address{Mathematical Sciences\\ University of Southampton\\ Highfield\\ Southampton SO17 1BJ\\ United Kingdom}
\email{b.koeck@soton.ac.uk}

\author{Lenny Taelman}
\address{KdVI -- Universiteit van Amsterdam\\ P.O. Box 94248\\ 1090 GE Amsterdam\\ the Netherlands}
\email{l.d.j.taelman@uva.nl}

\date{\today}

\begin{document}

\begin{abstract}
We use Grayson's binary multicomplex presentation of algebraic $K$-theory to give a new construction of exterior
power operations on the higher $K$-groups of a (quasi-compact) scheme.
We show that these operations satisfy the axioms of a $\lambda$-ring, including the product and composition laws.
To prove the composition law we show that the Grothendieck group of the exact category of integral polynomial
functors is the universal $\lambda$-ring on one generator.
\end{abstract}


\maketitle

\section*{Introduction}

Exterior powers of vector bundles over a scheme $X$ endow its Grothendieck group $K_0(X)$ with a
family of operations $\lambda^r: K_0(X) \rightarrow K_0(X)$, $r=0,1, \ldots$ These
$\lambda$-operations allow us to define Adams operations and the $\gamma$-filtration on $K_0(X)$ and are, more generally, at the
heart of Grothendieck's Riemann-Roch theory (see \cite{FultonLang}). This fundamental structure has
been extended to the higher $K$-groups $K_n(X)$, $n \ge 0$, using a variety of sophisticated
approaches and in various degrees of generality, by Kratzer \cite{Kratzer}, Hiller \cite{Hiller},
Grayson \cite{GraysonExterior}, Nenashev \cite{NenLambda} and Levine \cite{Levine}, and has been
most profoundly studied and applied in Soul\'e's seminal paper \cite{Soule}. Common to all these
constructions is that they use homotopy theory.

\medskip

In this paper we give a purely algebraic construction of the $\lambda$-operations on the higher
$K$-groups of any quasi-compact scheme $X$. Our construction is explicit in the following sense: in
his surprising paper \cite{Gray1}, Grayson has given explicit generators and relations for
$K_n(X)$, and our construction describes explicit (albeit intricate) images of these generators
under the $\lambda$-operations. Within the purely algebraic context of this paper, we prove moreover that our $\lambda$-operations satisfy the usual
axioms, including the product and composition laws. In a forthcoming paper we address the problem of matching up our $\lambda$-operations with Hiller's.

\medskip

To describe our results in more precise terms, we now recall the definition of a $\lambda$-ring.

\begin{defn*}
 A {\em pre-$\lambda$-ring} is a commutative unital ring $K$ together with maps $\lambda^r \colon K \rightarrow K$, $r \ge 1$, satisfying $\lambda^1(x) =x$ and the following axiom for all $x$, $y \in K$:
 \begin{enumerate}
  \item $\lambda^r (x + y) = \lambda^r(x) + \sum_{i=1}^{r-1} \lambda^{r-i}(x)\lambda^i (y)+ \lambda^r(y)$.
 \end{enumerate}
A  {\em $\lambda$-ring} $K$ is a pre-$\lambda$-ring satisfying the further axioms
 \begin{enumerate}
 \setcounter{enumi}{1}
  \item $\lambda^r(xy) = P_r (\lambda^1 (x), \dots, \lambda^r (x), \lambda^1 (y), \dots, \lambda^r (y))$
  \item $\lambda^r (\lambda^s (x)) = P_{r,s} (\lambda^1 (x) , \dots, \lambda^{rs} (x))$,
 \end{enumerate}
where $P_r$ and $P_{r,s}$ ($r$, $s > 0$) are certain universal integral polynomials (defined in such a way so the axioms (2) and (3) hold in every pre-$\lambda$-ring whose additive group is generated by elements $l$ with $\lambda^r (l)=0$ for all $r>1$ and in which products of elements of this type are again of this type; for details see \cite{FultonLang}).
\end{defn*}

Probably the most prominent example of a $\lambda$-ring is $K_0(X)$ (see \cite{FultonLang}). The object of this paper is to make $K_* (X)= \oplus_{n \ge 0} \; K_n(X)$ into a $\lambda$-ring.

\medskip

For each $n \ge 0$, Grayson \cite{Gray1} associates to an exact category $\mc{P}$ the exact category
$(B_\mathrm{b}^\mathrm{q})^n \mc{P}$ of so-called $n$-dimensional bounded acyclic binary complexes,
and proves that $K_n (\mc{P})$ is isomorphic to a relatively simple-to-describe quotient of the
Grothendieck group $K_0\left((B_\mathrm{b}^\mathrm{q})^n \mc{P}\right)$ (see
section~\ref{Binary multicomplexes and algebraic $K$-theory} for a detailed review of Grayson's
construction). Using the Dold--Puppe construction \cite{DoldPuppe}, we inductively construct functors
\[\Lambda_n^r: (B_\mathrm{b}^\mathrm{q})^n \mc{P}(X) \rightarrow (B_\mathrm{b}^\mathrm{q})^n \mc{P}(X)\]
for all $r, n >0$ from the usual exterior power functors
$\Lambda^r: \mc{P}(X) \rightarrow \mc{P}(X)$, $r \ge 0$, on the category $\mc{P}(X)$ of vector
bundles on $X$.

\medskip

The following theorems are the main results of this paper.

\medskip

{\bf Theorem A} [Theorem~\ref{main theorem for schemes}].
 {\em The functors
 $\Lambda^r_n$
 induce well-defined homomorphisms
 $\lambda^r \colon K_n (X) \rightarrow K_n (X)$
 for $r,n > 0$. }

\medskip

The tensor product induces the multiplication in the Grothendieck ring $K_0(X)$ and also an action of $K_0(X)$ on the higher $K$-groups $K_n(X)$. In particular, $K_* (X)= \oplus_{n \ge 0} \; K_n(X)$ carries the structure of a unital commutative ring in which the product of any two elements in $\oplus_{n \ge 1}\, K_n(X)$ is defined to be zero. Note that, if $n >0$, axiom~(1) for $x,y \in K_n(X)$ then follows from $\lambda^r:K_n(X) \rightarrow K_n(X)$ being a homomorphism (Theorem~A). Furthermore, the formula in axiom~(1) can be used to extend our
operations $\lambda^r \colon K_n (X) \rightarrow K_n (X)$, $n \ge 0$, to a
pre-$\lambda$-ring structure on $K_* (X)$.

\medskip

{\bf Theorem B} [Theorems~\ref{lambda axiom 2}, and \ref{lambda axiom 3}].
{\em The pre-$\lambda$-ring $K_*(X)$ is a $\lambda$-ring.}

\medskip

The first half of the paper is devoted to the construction of the
exterior power functors $\Lambda^r_n$.
Let $C_{\rm{b}}\mc{P}(X)$ denote the category of bounded complexes in $\mc{P}(X)$.
We use the Dold--Kan correspondence (reviewed along with the other necessary homological
preliminaries in section \ref{prelim}) to obtain a chain-homotopy invariant
functor $\Lambda^r_1 \colon C_{\rm{b}}\mc{P}(X) \rightarrow C_{\rm{b}}\mc{P}(X)$
for each $r > 0$; if $X$ is affine, then the bounded acyclic complexes in $\mc{P}(X)$
are precisely the contractible ones, so we obtain an endofunctor on the category of
bounded acyclic chain complexes in $\mc{P}(X)$. By generalising and iterating this procedure
over complexes of complexes we get the desired functors
$\Lambda^r_n \colon (B^{\rm{q}}_{\rm{b}})^n \mc{P}(X) \rightarrow (B^{\rm{q}}_{\rm{b}})^n \mc{P}(X)$.
This material is the subject of sections
\ref{Operations on acyclic complexes} and \ref{Operations on binary multicomplexes}.

\medskip

In the rather long section \ref{Simplicial tensor products} we construct a `simplicial tensor
product' $\otimes_{\Delta, n}$ on $(B^{\rm{q}}_{\rm{b}})^n \mc{P}(X)$. In defining exterior powers
on $K_0(X)$ we obtain from a short exact sequence of vector bundles
$0 \rightarrow V' \rightarrow V \rightarrow V'' \rightarrow 0$ a filtration of $\Lambda^r (V)$
whose successive quotients are $\Lambda^{r-i}(V') \otimes \Lambda^i (V'')$. We use our simplicial
tensor product of binary multicomplexes to obtain similar statements for short exact sequences in $(B_\mathrm{b}^\mathrm{q})^n \mc{P}(X)$;
our tensor product $\otimes_{\Delta, n}$ is to $\otimes$ as the exterior powers $\Lambda^r_n$ are to
$\Lambda^r$. The main result of the section (Proposition~\ref{tensor product vanishes}) is that the product induced by $\otimes_{\Delta, n}$
on $K_n(X)$ vanishes.

\medskip

In section \ref{proof of main theorem} we pass our exterior powers from the affine case
to general (quasi-compact) schemes and show that they induce well-defined operations
$\lambda^r \colon K_n(X) \rightarrow K_n(X)$.
As the product on $K_n(X)$ that is compatible with these operations is the zero product (by Proposition~\ref{tensor product vanishes}),
it follows that the $\lambda^r$ are group homomorphisms.

\medskip

In section \ref{Lambda axiom 4} we
show that that the resulting pre-$\lambda$-ring $K_*(X)$ satisfies the $\lambda$-ring axiom~(2) concerning products.

\medskip

The final $\lambda$-ring axiom (3) is proved in section \ref{Lambda final}. While the usual geometric splitting principle suffices to prove axiom~(2) for $K_*(X)$ (see section~\ref{Lambda axiom 4}) and both axioms~(2) and~(3) for $K_0(X)$, there seems to be no way of extending that approach to prove axiom~(3) for $K_*(X)$. We will rather proceed as follows.  As
$K_0(X)$ is a $\lambda$-ring, there exist short exact sequences in $\mc{P}(X)$ that prove the
relation $\lambda^r (\lambda^s (x)) = P_{r,s} (\lambda^1 (x) , \dots, \lambda^{rs} (x))$ in
$K_0(X)$ when $x$ is the class of a vector bundle $V$ on $X$. We will see (see
subsection~\ref{Proof of final axiom}) that if in fact these short exact sequences exist
functorially in $V$, then we can inductively prove the existence of short exact sequences in
$(B_\mathrm{b}^\mathrm{q})^n\mc{P}(X)$ that prove the relation above when $x$ is the class of an
object in $(B^\mathrm{q}_\mathrm{b})^n \mc{P}(X)$; in other words, we have then proved axiom~(3)
for $K_*(X)$.

\medskip

We are therefore reduced to showing the existence of such short exact sequences of
functors in $V$. This problem may be seen as a weak variant of the famous plethysm problem (see Remark~\ref{remark}). The
crucial insight now is that it becomes attackable when we also require these functors to be
polynomial (see Definition~\ref{PolynomialFunctor}). On the one hand, this requirement guarantees the existence of appropriate base change functors and hence reduces the
problem to $X=\operatorname{Spec}(\bZ)$ (see subsections~\ref{PolynomialFunctor} and~\ref{Proof of final axiom}). On the other hand, it makes the computation of the
corresponding Grothendieck group of functors feasible; this is the content of the following theorem,
which we highlight as it may be of independent interest.

\medskip

It is well-known that there exists a unique $\lambda$-ring structure on the ring $\bZ[s_1, s_2,\ldots]$ of integral polynomials in infinite variables such that $\lambda^r(s_1) = s_r$ for all $r$. Furthermore, let $\Pol_{< \infty}(\bZ)$ denote the category of polynomial functors over $\bZ$
of bounded degree (whose Grothendieck group is easily seen to be a pre-$\lambda$-ring).
\begin{thm*}[Theorem \ref{Computing Grothendieck group of PolFun}]
 The ring homomorphism
\[
    \bZ[s_1,s_2,\ldots ] \to K_0(\Pol_{< \infty}(\bZ)),\quad s_i \mapsto [\Lambda^i]
\]
is an isomorphism of pre-$\lambda$-rings.
\end{thm*}
This theorem obviously implies that the right-hand side is a $\lambda$-ring as well and hence that the short exact sequences of functors postulated above indeed exist. After interpreting polynomial functors as modules over certain Schur algebras following Krause \cite[Section~8.2]{Krause} (see subsection~\ref{SchurAlgebra}), we will prove the theorem by following Serre's computation \cite{Serre} of the Grothendieck group of representations of the group scheme $\mathrm{GL}_{n,\bZ}$ (see subsection~\ref{Grothendieck groups over Z}). A crucial ingredient here is Green's computation \cite{Gr} of the Grothendieck group of polynomial functors over a field.

\medskip

The fundamental idea of proving $\lambda$-ring axioms for Grothendieck groups of complexes via the corresponding axioms for a Grothendieck group of appropriate functors is also sketched in an exchange of letters between Deligne and Grothendieck \cite{DG67}. Both their correspondence and the introduction of Serre's paper \cite{Serre} already allude to a role of Serre's result for $\lambda$-operations.

\medskip

In a forthcoming paper we will complement the somewhat intricate constructions of this paper with simpler formulae that (help to) compute our $\lambda$-operations in certain cases. For instance, we will give formulae for our $\lambda$-operations when applied to $K_1$-groups of rings or to external products $K_m(X) \times K_n(X) \rightarrow K_{m+n}(X)$.

\medskip

{\bf Acknowledgements}. The first two authors thank Dan Grayson and Marco Schlichting for stimulating discussions and emails about exterior powers of binary complexes and for valuable comments and suggestions relating to earlier drafts. The third author wishes to thank Wilberd van der Kallen for helpful elucidations about polynomial functors and integral representation theory. The second author moreover thanks Christophe Soul\'e for sending a photocopy of the correspondence \cite{DG67} upon posting this paper on arXiv. Finally all three authors thank the referee for carefully reading the paper and for many helpful suggestions.

\section{Binary multicomplexes and algebraic \texorpdfstring{$K$}{K}-theory}
\label{Binary multicomplexes and algebraic $K$-theory}

In this section we review the description of algebraic $K$-groups in terms of binary complexes
given in \cite{Gray1}. We also prove a simple lemma about shifted binary complexes to justify a
slight modification of Grayson's description. The lemma is also useful for computations.

\medskip

Recall that an \emph{exact category} in the sense of Quillen \cite{Quil1} is an additive category
with a distinguished class of `short exact sequences' that behave like the short exact sequences of
an abelian category. A small exact category $\mc{N}$ may also be thought of
as a full subcategory of an ambient abelian category $\mc{A}$ such that $\mc{N}$ is closed under
extensions in $\mc{A}$.
\footnote{This is the \emph{Gabriel--Quillen embedding theorem \cite{Thomason}, A.7.1, A.7.16.}}
The category of chain complexes in an exact category is again an exact
category, with short exact sequences defined to be those sequences of chain maps that are short
exact in each degree. In this paper we consider only complexes that are concentrated in
non-negative degrees, those with an underlying $\bZ_{\ge 0}$-graded object. We denote this
category of chain complexes in $\mc{N}$ by $C\mc{N}$. A chain complex is bounded if it has only
finitely many non-zero objects. The exact subcategory of $C\mc{N}$ of bounded chain complexes is
denoted $C_{\rm{b}}\mc{N}$. An \emph{acyclic} complex in an exact category is a chain complex
$N_{\cd}$ in $\mc{N}$ whose differentials $d_i \colon N_i \rightarrow N_{i-1}$ factor as
$N_i \rightarrow Z_i \rightarrow N_{i-1}$ (with $Z_i$ in $\mc{N}$), such that each
$0 \rightarrow Z_{i+1} \rightarrow N_i \rightarrow Z_i \rightarrow 0$
is a short exact sequence in $\mc{N}$.
\footnote{\label{fn} This is not in general the same thing as being a long exact sequence in the
ambient abelian category $\mc{A}$. However in this paper we work only with
\emph{idempotent complete} exact categories, in which case the two notions coincide.
See \cite{Gray1} and \cite{Thomason} A.9.2.}
The full subcategories of acyclic complexes in $C\mc{N}$ and $C_{\rm{b}}\mc{N}$ are also exact, and are
denoted $C^{\rm{q}}\mc{N}$ and $C^{\rm{q}}_{\rm{b}}\mc{N}$.\\
Since each of these categories of complexes is also an exact category, we can iterate their
construction to define \emph{$n$-dimensional multicomplexes} in $\mc{N}$. A $1$-dimensional
multicomplex in $\mc{N}$ is simply a chain complex, an object of $C\mc{N}$. An $n+1$-dimensional
multicomplex in $\mc{N}$ is a chain complex in the exact category $C^n\mc{N}$ of
$n$-dimensional multicomplexes in $\mc{N}$. We define categories of bounded and/or acyclic
multicomplexes, $(C_{\rm{b}})^n\mc{N}$, $(C^{\rm{q}})^n\mc{N}$ and
$(C^{\rm{q}}_{\rm{b}})^n\mc{N}$ analogously.
With these notions in place, we can define binary complexes and multicomplexes.

\begin{defn}[]\
\begin{enumerate}
 \item A \emph{binary complex} in an exact category $\mc{N}$ is a triple $(N_{\cd},d,\tilde{d})$
 consisting of: a $\bZ_{\ge 0}$-graded collection of objects of $\mc{N}$ together with two
 differentials $d$ and $d'$ such that $(N_{\cd},d)$ and $(N_{\cd},\tilde{d})$ are chain complexes
 in $\mc{N}$. A binary complex can be regarded as pair of objects of $C\mc{N}$ that have the same
 underlying graded object. A morphism of binary complexes is a degree 0 map between these
 underlying objects that commutes with both differentials. The category of binary complexes in
 $\mc{N}$ is denoted $B\mc{N}$. This is an exact category in the same way that $C\mc{N}$ is.
 \item A \emph{bounded acyclic} binary complex in $\mc{N}$ is a binary complex such that the chain
 complexes $(N_{\cd},d)$ and $(N_{\cd},\tilde{d})$ are bounded and acyclic. The category of bounded
 acyclic binary complexes in $\mc{N}$ is denoted $B_{\rm{b}}^{\rm{q}}\mc{N}$. It is an exact
 subcategory of $B\mc{N}$.
 \item An \emph{$n$-dimensional binary multicomplex} is an object of the exact category
 $B^n\mc{N} = B\cdots B\mc{N}$ (defined in the same way as $C^n\mc{N}$). An $n$-dimensional
 bounded acyclic binary multicomplex is an object of $(B^{\rm{q}}_{\rm{b}})^n\mc{N}$.
\end{enumerate}
\end{defn}

\begin{rem}
\label{commute}
 A (bounded acyclic) binary multicomplex $N_{\cd}$ of dimension $n$ is equivalent to the following
 data: a (bounded) $\bZ_{\ge 0}^n$-graded collection of objects of $\mc{N}$ equipped with
 \textit{two} (acyclic) differentials, denoted $d^i$ and $\tilde{d}^i$, in each direction
 $1 \le i \le n$ subject to the following commutativity requirements:
 \begin{enumerate}
  \item $d^id^j = d^jd^i$
  \item $d^i\tilde{d}^j = \tilde{d}^jd^i$
  \item $\tilde{d}^id^j = d^j\tilde{d}^i$
  \item $\tilde{d}^i\tilde{d}^j = \tilde{d}^j\tilde{d}^i$
 \end{enumerate}
 whenever $i \neq j$.\\
 Another way to look at these commutativity restraints is that the various subsets of the
 differentials form (non-binary) multicomplexes: for each $i = 1, \dots, n$, choose $d^i$ or
 $\tilde{d}^i$, and consider the object that has the same underlying $\bZ^n$-graded object
 as $N_{\cd}$, but now has \textit{one} acyclic differential in each direction $i$, given
 by $d^i$ or $\tilde{d}^i$, depending on our choice. For each of the $2^n$ choices of
 differentials, the resulting object is a bounded acyclic multicomplex, i.e.\@, an object of
 $(C^{\rm{q}}_{\rm{b}})^n\mc{N}$; conversely, given a pair of differentials $d^i,\tilde{d}^i$
 in each direction, if the $2^n$ choices all form objects of $(C^{\rm{q}}_{\rm{b}})^n\mc{N}$,
 then the whole assembly is an object of $(B^{\rm{q}}_{\rm{b}})^n\mc{N}$.
\end{rem}

Since this category of bounded acyclic binary complexes in $\mc{N}$ is itself an exact category, we
can form its Grothendieck group $K_0(B^{\rm{q}}_{\rm{b}}\mc{N})$. The main theorem of
\cite{Gray1} is a surprising connection between this group and the $n^{\rm{th}}$ higher $K$-group
of $\mc{N}$. We call an $n$-dimensional binary multicomplex \emph{diagonal} if the pair of
differentials in some direction are equal, i.e, if $d^i = \tilde{d}^i$ for some
$1 \le i \le n$. Grayson's theorem, which we shall hereafter use as our definition of the $K$-groups,
says that $K_n(\mc{N})$ is isomorphic to the quotient of the Grothendieck group of
$B^{\rm{q}}_{\rm{b}}\mc{N}$ by the subgroup generated by the classes of the diagonal bounded
acyclic binary multicomplexes. More formally:

\begin{thmdef}[\cite{Gray1}, Corollary 7.4]
\label{Grayson's def}
 For $\mc{N}$ an exact category and $n \ge 0$, the abelian group $K_n(\mc{N})$ is presented as
 follows. There is one generator for each bounded acyclic binary multicomplex of dimension $n$, and
 there are two families of relations:
 \begin{enumerate}
 \item $[N'] + [N''] = [N]$ if there is a short exact sequence
 \[
  0 \rightarrow N' \rightarrow N \rightarrow N'' \rightarrow 0
 \]
 in $(B^{\rm{q}}_{\rm{b}})^n\mc{N}$, and
 \item $[D] = 0$ if $D$ is a diagonal bounded acyclic binary multicomplex.
 \end{enumerate}
\end{thmdef}

 We remark that our statement of Theorem / Definition~\ref{Grayson's def} is subtly different than the one  originally given by Grayson. Our bounded acyclic binary multicomplexes are first-quadrant
 multicomplexes, those that are supported in $\bZ_{\ge 0}^n$, whereas Grayson's do not have to
 satisfy this condition. The absolute lower bound for complexes is a technical constraint that we
 need in order to use the Dold--Kan correspondence. Our additional condition is harmless as the following proposition shows. For this, let $K_n^\mathrm{Gr}(\mathcal{N})$ temporarily denote the $n^\mathrm{th}$ $K$-group of $\mathcal{N}$ as defined in \cite{Gray1}.

\begin{prop}
For every exact category $\mathcal{N}$ and every $n \ge 0$, the canonical homomorphism $K_n(\mathcal{N}) \rightarrow K_n^\mathrm{Gr}(\mathcal{N})$ is bijective.
\end{prop}

\begin{proof}
For ease of presentation we shall prove this for $n=1$ only: there is no additional difficulty for $n >1$. Let $B^\mathrm{q}_\infty \mathcal{N}$ denote the category of bounded acyclic binary complexes in $\mathcal{N}$ that may be supported anywhere on~$\bZ$. For $i \ge 0$, let $B^\mathrm{q}_{\ge -i}\mathcal{N}$ denote the full subcategory of $B^\mathrm{q}_\infty \mathcal{N}$ consisting of complexes that are supported on $[-i, \infty]$. We then have $\cup_i B^\mathrm{q}_{\ge -i} \mathcal{N} = B^\mathrm{q}_\infty \mathcal{N}$ and hence $\lim_i K_0(B^\mathrm{q}_{\ge -i} \mathcal{N}) = K_0(B^\mathrm{q}_\infty \mathcal{N})$. Let $T_i$ denote the subgroup of $K_0(B^\mathrm{q}_{\ge -i} \mathcal{N})$ generated by diagonal complexes and let $T$ denote the similarly defined subgroup of $K_0(B^\mathrm{q}_\infty \mathcal{N})$. The resulting injective homomorphism $\lim_i T_i \rightarrow T$ is also surjective because all complexes are assumed to be bounded. We therefore obtain an ismorphism
\[\lim_i \left(K_0 (B^\mathrm{q}_{\ge -i} \mathcal{N})/ T_i \right) \cong \lim_i K_0 (B^\mathrm{q}_{\ge -i}\mathcal{N})/ \lim_i T_i \cong K_0(B^\mathrm{q}_\infty \mathcal{N})/T = K_1^\mathrm{Gr} (\mathcal{N}).\]
The following lemma (after generalising it from $B^\mathrm{q}_{\ge 0} \mathcal{N}$  to $B^\mathrm{q}_{\ge -i} \mathcal{N}$) shows that, for every $i \ge 0$, `shifting' induces a two-sided inverse to the negative of the canonical homomorphism $K_0(B^\mathrm{q}_{\ge -i} \mathcal{N})/T_i \rightarrow  K_0(B^\mathrm{q}_{\ge -i-1} \mathcal{N})/T_{i+1}$. Hence the canonical map
\[K_1 (\mathcal{N}) = K_0(B^\mathrm{q}_{\ge 0} \mathcal{N})/T_0 \rightarrow \lim_i \left(K_0(B^\mathrm{q}_{\ge -i} \mathcal{N})/T_i \right) \cong K_1^\mathrm{Gr}(\mathcal{N})\]
is an isomorphism, as was to be shown.
\end{proof}

\begin{defn}
 Let $N_{\cd}$ be an acyclic binary complex with differentials $d$ and $\tilde{d}$. The
 $k^{\textnormal{th}}$ \emph{shift} of $N$, denoted $N[k]$, is the acyclic binary complex that has
 the same collection of objects as $N$, but `shifted' $k$ places, i.e.\@, $(N[k])_i = N_{i-k}$, and
 differentials given by $(-1)^kd$ and $(-1)^k\tilde{d}$.
\end{defn}

\begin{lem}
 \label{shifting}
 For any bounded acyclic binary complex $N_{\cd}$ and $k \in \mathbb{Z}_{\ge 0}$, we have
 $[N[k]] = (-1)^k[N]$ in $K_1\mc{N}$.
\end{lem}

\begin{proof}
 It is enough to show that $ [N[1]] = -[N]$. There is a short exact sequence
 \[
  0 \rightarrow N_{\cd} \rightarrow \operatorname{cone}(N_{\cd}) \rightarrow N_{\cd}[1] \rightarrow 0,
 \]
 where $\operatorname{cone}(N_{\cd})$ denotes the mapping cone of the identity map
 $N_{\cd} \stackrel{1}{\rightarrow} N_{\cd}$ ($\operatorname{cone}(N_{\cd})$ is a binary complex
 in the obvious way). So it suffices to show that $\operatorname{cone}(N_{\cd})$ vanishes in
 $K_1\mc{N}$. Let $N_n$ be the left-most non-zero object of $N_{\cd}$, and let
 $\operatorname{trun}(N_{\cd})$ be the (not necessarily acyclic) binary complex formed by
 truncating $N_{\cd}$ to forget $N_n$: that is, $\operatorname{trun}(N_{\cd})$ has a $0$ in place
 of $N_n$. Then there is a short exact sequence
 \[
  0 \rightarrow \operatorname{cone}(\operatorname{trun}(N_{\cd})) \rightarrow
  \operatorname{cone}(N_{\cd}) \rightarrow \Delta(N_n \stackrel{1}{\rightarrow} N_n)
  \rightarrow 0,
 \]
 where $\Delta(N_n \stackrel{1}{\rightarrow} N_n)$ is the diagonal binary complex
 \[
  \xymatrix{
  0 \binr{}{}
  & N_n \binr{1}{1}
  & N_n \binr{}{}
  & 0,
  }
 \]
 which is supported in degrees $n+1$ and $n$. Mapping cones of identities are always acyclic, so
 $\operatorname{cone}(\operatorname{trun}(N_{\cd}))$ is acyclic even when
 $\operatorname{trun}(N_{\cd})$ is not. Since $\Delta(N_n \stackrel{1}{\rightarrow} N_n)$ is
 diagonal its class vanishes in $K_1\mc{N}$, so the above short exact sequence yields the
 relation $[\operatorname{cone}(N_{\cd})] = [\operatorname{cone}(\operatorname{trun}(N_{\cd}))]$.
 We iterate this procedure by repeatedly truncating $\operatorname{trun}(N_{\cd})$ to show that
 $[\operatorname{cone}(N_{\cd})]$ is zero.
\end{proof}

The same proof gives the analogous result for binary multicomplexes:
for $N$ in $(B^{\rm{q}}_{\rm{b}})^n \mc{N}$
the class of $N$ shifted one place in any of the $n$ possible directions
in $K_n (\mc{N})$ is $-[N]$. From this the actions of more general shifts
(in multiple directions) follow immediately.

\section{Preliminaries from homological algebra}
\label{prelim}

In this section we recall some preliminaries from the homological algebra of exact categories.
We say what it means for an exact category to be idempotent complete or split, and show that
the notions of acyclicity and contractibility of complexes coincide in exact categories
that have both of these properties.
We then review simplicial objects and the Dold--Kan correspondence.
Finally we discuss functors of finite degree, a weakening of the concept of additive functors.
These three topics may seem rather disjoint here, but we bring them together in the next section
to produce functors between categories of chain complexes that preserve boundedness and acyclicity,
paving the way for a functor on binary multicomplexes that induces a map on $K$-theory.

\begin{defn}
 An exact category $\mc{N}$ is \emph{idempotent complete}
 if every idempotent endomorphism in $\mc{N}$ the category has a kernel.
\end{defn}

This does not hold for example for the category of free modules over a ring when there exists a non-free projective module. All of the exact categories we use in this paper are idempotent complete. This is an assumption
on the `base level' exact categories we introduce, but will need to be proven for categories of
multicomplexes (Lemma~\ref{iterate lemma}). Idempotent complete exact categories come with an
embedding into an abelian category $\mc{N} \hookrightarrow \mc{A}$ that
\emph{supports long exact sequences}: a chain complex is acyclic in $\mc{N}$ if and only if it is
exact when considered as a chain complex of $\mc{A}$ (see \S1 of \cite{Gray1}).
Homological algebra is therefore quite straightforward in idempotent complete exact categories.\\
Contractible complexes in idempotent complete exact categories are always acyclic; indeed
this is an equivalent characterisation of idempotent completeness (\cite{Buh}, 10.9).
Acyclic complexes in exact categories (even idempotent complete ones)
are not usually contractible. There is a useful criterion for contractibility, however.
Recall that a chain complex $(C.,d)$ is called \emph{split}
if there exist maps $s_n\colon C_{n-1} \rightarrow C_n$ such that
$d_ns_nd_n = d_n$.

\begin{lem}
\label{acyclic and split}
 A chain complex in an idempotent complete exact category is contractible if and only if it is
 acyclic and split.
\end{lem}

\begin{proof}
It follows the definition of a chain homotopy that
contractible complexes in idempotent complete exact categories are also split.
Conversely, an elementary argument shows that if a complex in an exact category is acyclic
and split, then the collection of splitting maps $\{s_n\}$ describes a homotopy from its identity
map to its zero map.
\end{proof}

If an acyclic complex is split, each of the constituent short exact sequences that it factors
into is split: that is, isomorphic to a canonical direct sum sequence (the converse is obviously true as well).
Recall that an exact category is called \emph{split exact} if all of its declared short exact sequences are split.
In such an exact category, all acyclic complexes are split.
Hence the notions of contractibility and acyclicity coincide for complexes in a split exact
category that is also idempotent complete. An example of such an exact category is the category
$\proj (R)$ of (finitely generated) projective modules over a ring $R$. That acyclic complexes are
contractible in this category is key to the results of this paper.

\medskip

We now turn to the Dold--Kan correspondence. To give its statement we need the language of
simplicial objects. Recall that $\Delta$ denotes the \emph{simplex category}: the category whose
objects are the finite non-empty ordered sets $[n] = \{ 0 < 1 < \cdots < n \}$ and whose morphisms are
the order-preserving maps. A \emph{simplicial object} in a category $\mc{C}$ is a contravariant functor
from $\Delta$ to $\mc{C}$, and the natural transformations between such functors make
$\mc{C}^{\Delta^{\rm{op}}}$ into a category. Equivalently, a simplicial object $C$ in $\mc{C}$ can
be specified to be a collection of objects $C_n$, $n \in \bN$, of $\mc{C}$ together with face
maps $\delta_i \colon C_n \rightarrow C_{n-1}$ and degeneracy maps
$\sigma_j \colon C_n \rightarrow C_{n+1}$, $i,j=0,\dots,n$, satisfying various combinatorial
identities. A morphism between simplicial objects $C$ and $D$ is a collection of morphisms
$C_n \rightarrow D_n$ that commutes with the faces and degeneracies. A homotopy
$h \colon f \simeq g$ between simplicial maps $f,g \colon C \rightarrow D$ is a simplicial morphism $h \colon C \times \Delta^1 \rightarrow D$ (where $\Delta^1$ denotes the simplicial set corresponding to the ordered set $\{0<1\}$, as usual) such that $h|_{C \times \{0\}}= f$ and $h|_{D\times \{1\}} = g$; it can also be described as collection of morphisms $h_i \colon C_n \rightarrow D_{n+1}$, $i=0,\dots,n$, which satisfy further combinatorial identities determined by compositions relating $f$, $g$, the $h_i$, and the faces and degeneracies of $C$ and $D$. See, for example, Chapter 8 of \cite{WeiHom} for full definitions of simplicial objects and homotopies.

\medskip

If $F \colon \mc{C} \rightarrow \mc{D}$ is a covariant functor, then post-composition with $F$
induces a functor between categories of simplicial objects
$\mc{C}^{\Delta^{\rm{op}}} \rightarrow \mc{D}^{\Delta^{\rm{op}}}$. Abusing notation, we shall also
call this functor $F$. Importantly, if $h \colon f \simeq g$ is a simplicial homotopy between
$f,g \colon C \rightarrow D$, then $F(h) \colon F(f) \simeq F(g)$ is a simplicial homotopy between
$F(f), F(g) \colon F(C) \rightarrow F(D)$. The analogous statement for chain homotopies between
chain maps is not true if $F$ is not additive. The Dold--Kan correspondence shows that chain
complexes and simplicial objects are equivalent in a non-obvious way, and allows us to induce
homotopy-preserving functors between categories of chain complexes, even when the original functors
are not additive.

\begin{defn}
\label{simp}
 Let $\mc{P}$ be an additive category.
 Given a chain complex $C. \in C\mc{P}$, we define a simplicial object
 $\Gamma(C.) \in \mc{P}^{\Delta^{\rm{op}}}$ as follows.
 \begin{enumerate}
  \item {\bf Objects:} Given $p \le n$, let $\eta$ range over all surjections
  $[n] \twoheadrightarrow [p]$ in $\Delta$, and let $C_p\langle\eta\rangle $ denote a copy of
  $C_p$ that is labelled by $\eta$. For each~$n$, set
   \[
      \Gamma(C)_n \defeq \bigoplus_{p \le n}\bigoplus_{\eta}C_p\langle \eta \rangle.
   \]
   \item {\bf Maps:} If $\alpha \colon [m] \rightarrow [n]$ is a morphism in $\Delta$, we describe
   $\Gamma(\alpha)$ by describing each $\Gamma (\alpha, \eta)$, the restriction of
   $\Gamma( \alpha)$ to the summand $C_p \langle \eta \rangle$ of $\Gamma(C)_n$. Let
   \[
      [m] \stackrel{\eta '}{\twoheadrightarrow} [q] \stackrel{\varepsilon}{\hookrightarrow} [p]
   \]
   be the unique epi-monic factorisation of $\eta \alpha$. Then
   \begin{equation*}
    \Gamma (\alpha, \eta) \defeq
    \begin{cases}
    1 \colon C_p\langle \eta \rangle \rightarrow C_p\langle \eta '\rangle & \text{if $q = p$},\\
    d_p \colon C_p \langle \eta \rangle \rightarrow C_{p-1} \langle \eta ' \rangle &
    \text{if $q = p - 1$ and $\varepsilon = \varepsilon_p$},\\
    0 & \text{otherwise}.
    \end{cases}
   \end{equation*}
 \end{enumerate}
 This construction extends to a functor\footnote{Other authors
 (e.g.\@, \cite{WeiHom}) use $K$ in place of $\Gamma$; we avoid this notation for obvious reasons.}
 $\Gamma \colon C\mc{P} \rightarrow \mc{P}^{\Delta^{\rm{op}}}$.
\end{defn}

\begin{thm*}[Dold--Kan correspondence]
 If $\mc{P}$ is idempotent complete, then the functor
 $\Gamma \colon C\mc{P} \rightarrow \mc{P}^{\Delta^{\rm{op}}}$ is an equivalence of categories.
 Furthermore, $\Gamma$ is exact and preserves homotopies.
\end{thm*}

\begin{proof}
 Chapter 8 of \cite{WeiHom} proves this when $\mc{P}$ is an abelian category.
 The general case is \S1.2.3 of \cite{HigherAlg}.
\end{proof}

The inverse functor to $\Gamma$ is most simply described for an abelian category.
\begin{defn}
\label{moore complex}
 Let $A$ be a simplicial object in an abelian category $\mc{A}$.
 \begin{enumerate}
  \item The \emph{associated chain complex} $C(A)$ has objects $C(A)_n = A_n$ and differential
  \[ d_n = \sum_{i=0}^n(-1)^i\delta_i \colon C(A)_n \rightarrow C(A)_{n-1}. \]
  \item The subcomplex
 \[
  D(A)_n = \sum_{i=0}^n \operatorname{Im} (\sigma_i \colon A_{n-1} \rightarrow A_n)
 \]
 is called the \emph{degenerate subcomplex} of $C(A)$.
  \item The \emph{normalized Moore complex} $N(A)$ has objects
  \[ N_n(A) = A_n / D(A)_n \]
 with the induced differential $\bar{d}_n$.
 \end{enumerate}
 The associated chain complex splits globally as $C(A) = N(A) \oplus D(A)$.
\end{defn}
The normalized Moore complex defines a functor
$N\colon \mc{A}^{\Delta^{\rm{op}}} \rightarrow C\mc{A}$.
It is exact and preserves homotopies, and is inverse to $\Gamma$ (up to natural isomorphism).
Now if $\mc{P}$ is an idempotent complete exact category, then there is an embedding
$\mc{P} \subseteq \mc{A}$ into an abelian category such that $\mc{P}$ is closed under taking direct
summands in $\mc{A}$. If $P$ is an object of
$\mc{P}^{\Delta^{\rm{op}}} \subseteq \mc{A}^{\Delta^{\rm{op}}}$, then the associated chain complex
$C(P)$ is a chain complex $\mc{A}$ with objects in $\mc{P}$. But $N(P)$ is a direct summand of
$C(P)$, which has objects in $\mc{P}$, so $N(P)$ has objects in $\mc{P}$. Therefore $N$ restricts
to a functor $\mc{P}^{\Delta^{\rm{op}}} \rightarrow C\mc{P}$. Furthermore the functor $N$ is
exact and preserves homotopies. See \cite{HigherAlg} for further details.\\
We conclude our preliminaries by discussing functors of finite degree.

\begin{defn}
 Let $F\colon \mc{C} \rightarrow \mc{D}$ be any functor between additive categories that satisfies
 $F(0) = 0$. Then there is a functorial decomposition
 \[
    F(X \oplus Y) = F(X) \oplus \operatorname{cr}_2(F)(X,Y) \oplus F(Y),
 \]
 where $ \operatorname{cr}_2(F) \colon \mc{C} \times \mc{C} \rightarrow \mc{D}$ is the second
 \emph{cross-effect functor} (see \cite{EM_cross_effects}), which is defined to be the kernel of
 the natural projection $F(X \oplus Y) \rightarrow F(X) \oplus F(Y)$. The functor $F$ is said to
 have \emph{degree} $\le 1$ if it is additive (i.e.\@, if $\operatorname{cr}_2(F)$ vanishes), and
 we say that $F$ has degree $\le d$ if $\operatorname{cr}_2(F)(X,Y)$ is of degree $\le d-1$ in
 each argument. If $F$ is of degree $\le d$, then $F$ is of degree $\le d'$ for all $d' \ge d$.
 We say that $F$ \emph{has degree} $d$ if it has degree $\le d$ but does not have degree $\le d-1$.
\end{defn}

\begin{ex}
 For $R$ a non-zero commutative ring,
 the exterior power $\Lambda^r\colon \proj(R) \rightarrow \proj(R)$ has degree $r$ for each $r > 0$.
 This follows from the canonical decomposition
 \[
 \Lambda^r(X \oplus Y) \cong \Lambda^r(X) \oplus
 \left( \bigoplus_{i=1}^{r-1} \Lambda^{r-i}(X) \otimes \Lambda^i(Y) \right)\oplus \Lambda^r(Y).
 \]
\end{ex}

If $F\colon \mc{P} \rightarrow \mc{Q}$ is an \textit{additive} functor between exact categories,
and if $P_{\cdot}$ is a bounded complex, then $NF\Gamma(P_{\cdot})$ is certainly bounded again.
This also holds true for functors of finite degree, as the following lemma shows.

\begin{lem}[\cite{SatK}, Corollary 4.6]
\label{length}
 Let $P_{\cd}$ be a chain complex in $C\mc{P}$ of length $\ell$, and let
 $F\colon \mc{P} \rightarrow \mc{Q}$ be a functor of degree $d$ between exact categories. Then
 $NF\Gamma(P_{\cd})$ has length less than or equal to $d \ell$. \qed
\end{lem}

\section{Operations on acyclic complexes}
\label{Operations on acyclic complexes}

In this rather abstract section we describe how to use the Dold--Kan correspondence to extend a
functor $F \colon \mc{P} \rightarrow \mc{P}$ on an idempotent complete exact category to a functor
on each category of multicomplexes $F_n \colon C^n\mc{P} \rightarrow C^n\mc{P}$, $n \ge 1$.
We show that if $\mc{P}$ is split exact, then the extended functors $F_n$ send acyclic multicomplexes
to acyclic multicomplexes. We also show that if $F$ is of finite degree, then each $F_n$
preserves bounded multicomplexes and is also of finite degree.

\begin{prop}
 \label{F acyclic}
 Let $F \colon \mc{P} \rightarrow \mc{P}$
 be a covariant functor on an idempotent complete exact category, with $F(0) = 0$.
 Let $F_1 \defeq NF\Gamma \colon C\mc{P} \rightarrow C\mc{P}$ denote the induced functor.
 Then the following statements are true:

 \begin{enumerate}
  \item $F_1(0) = 0$.
  \item $F_1$ sends contractible complexes to contractible complexes.
  \item If $\mc{P}$ is split exact, then $F_1$ sends acyclic complexes to acyclic complexes.
  \item If $F$ is of degree at most $d$, then $F_1$ sends bounded complexes to bounded complexes and
  $F_1$ is again of degree at most $d$.
 \end{enumerate}
\end{prop}

\begin{proof}
 Part $(1)$ is trivial.\\
 For part $(2)$, the functors $\Gamma \colon C\mc{P} \rightarrow \mc{P}^{\Delta^{\rm{op}}}$ and
 $N\colon \mc{P}^{\Delta^{\rm{op}}} \rightarrow C\mc{P}$ preserve homotopies and send $0$ to $0$,
 so they both send contractible objects to contractible objects.
 Furthermore, $F$ sends homotopies in $\mc{P}^{\Delta^{\rm{op}}}$ to homotopies in
 $\mc{P}^{\Delta^{\rm{op}}}$---if $h\colon f \sim g$ is a homotopy,
 then $F(h)\colon F(f) \sim F(g)$ is a homotopy. Since $F$ also has the property that $F(0) = 0$,
 we see that if $A \simeq 0$ in $\mc{P}^{\Delta^{\rm{op}}}$, then $F(A) \simeq F(0) = 0$.
 Therefore $NF\Gamma(P_{\cd})$ is contractible in $C\mc{P}$.\\
 Following Lemma~\ref{acyclic and split}, the acyclic complexes in a split exact idempotent
 complete exact category coincide with the contractible ones, so part (3) follows from part (2).\\
 Finally we consider part $(4)$. The first part of this statement is Lemma~\ref{length}.
 For the second part we note that, since $N$ and $\Gamma$ are additive, it is enough to show that
 $F \colon \mc{A}^{\Delta^{\rm{op}}} \rightarrow \mc{B}^{\Delta^{\rm{op}}}$ is of degree
 $\le d$. This is proven by induction on $d$.
\end{proof}

Part $(3)$ of Proposition~\ref{F acyclic} may not hold in an exact category that is not split exact,
as is shown in the following example.

\begin{ex}
 Let $F$ be the degree $2$ endofunctor $A \mapsto A^{\otimes 2}$ on the abelian category of
 abelian groups, and let $C_{\cd}$ be the short exact sequence
 \[ 0 \longrightarrow \bZ \stackrel{2}{\longrightarrow} \bZ
 \longrightarrow \bZ/2\bZ \longrightarrow 0 \]
 considered as an acyclic complex concentrated in degrees 0, 1 and 2.
 Then $NF\Gamma(C_{\cd}) = N\operatorname{diag}(\Gamma(C_{\cd}) \otimes \Gamma(C_{\cd}))$
 is homotopy equivalent to $\operatorname{Tot}(C_{\cd} \otimes C_{\cd})$
 by the Eilenberg--Zilber theorem (\cite{May}, \S 29)\footnote{
 See also Definition~\ref{simplicial tensor product}, Lemma~\ref{comparison}.}.
 But the homology group $H_2(\operatorname{Tot}(C_{\cd} \otimes C_{\cd}))$ is $\bZ/2\bZ$,
 so $NF\Gamma(C_{\cd})$ is not exact. Furthermore, the short exact sequence of functors
 \[
  0 \longrightarrow N\Lambda^2\Gamma \longrightarrow NF\Gamma
  \longrightarrow N\operatorname{Sym}^2\Gamma \longrightarrow 0
 \]
 shows that at least one of $N\Lambda^2\Gamma(C_{\cd})$ or $N\operatorname{Sym}^2\Gamma(C_{\cd})$
 is not exact either.
\end{ex}

We now iterate the Dold--Kan correspondence to describe induced functors on categories of acyclic
multicomplexes.

\begin{defn}
\label{operations on multicomplexes}
 Let $F\colon \mc{P} \rightarrow \mc{P}$ be a covariant functor on an idempotent complete
 exact category. We define functors
 \[
  F_n\colon C^n\mc{P} \rightarrow C^n\mc{P}
 \]
 for all $n \ge 0$ recursively as follows:
 \begin{enumerate}
  \item $F_0 \defeq F \colon \mc{P} \rightarrow \mc{P}$.
  \item By regarding an object of $C^{n+1}\mc{P}$ as a chain complex
  in the exact category $C^n\mc{P}$, we define
  $F_{n+1} \defeq NF_n\Gamma$.
 \end{enumerate}
\end{defn}

To show that $F_n$ sends acyclic multicomplexes to acyclic multicomplexes in a nice exact category,
we need to know that $(C^{\rm{q}})^n\mc{P}$ satisfies the same hypotheses as $\mc{P}$. This is the
content of the following technical lemma. The proof is not enlightening for the rest of the paper,
so we relegate it to appendix \ref{acyclic app}.

\begin{lem}
\label{iterate lemma}
 Let $\mc{P}$ be an exact category. For all $n > 0$ we have the following:
 \begin{enumerate}
  \item If $\mc{P}$ is idempotent complete, then $C^n\mc{P}$ and $(C^{\rm{q}})^n\mc{P}$ are also
  idempotent complete.
  \item If $\mc{P}$ is split exact, then $(C^{\rm{q}})^n\mc{P}$ is also split exact.
 \end{enumerate}
 The analogous results for the categories $C^n_{\rm{b}}\mc{P}$ and $(C^{\rm{q}}_{\rm{b}})^n\mc{P}$
 of bounded multicomplexes also hold.
\end{lem}

\begin{cor}
\label{operations restrict}
 Let $\mc{P}$ be a split exact idempotent complete exact category,
 and $F\colon \mc{P} \rightarrow \mc{P}$ a covariant functor such that $F(0)=0$.
 Then for $n \ge 0$ the functors of Definition~\ref{operations on multicomplexes}
 restrict to functors
 \[
  F_n\colon (C^{\rm{q}})^n\mc{P} \rightarrow
  (C^{\rm{q}})^n\mc{P}.
 \]
 Furthermore, if $F$ is of finite degree, then $F_n$ sends bounded multicomplexes to bounded
 multicomplexes. That is, each $F_n$ restricts to a functor
 \[
  F_n\colon (C^{\rm{q}}_{\rm{b}})^n\mc{P} \rightarrow
  (C^{\rm{q}}_{\rm{b}})^n\mc{P}.
 \]
\end{cor}

\begin{proof}
 We consider the unbounded case first.
 By Proposition~\ref{F acyclic} $(1)$, we easily see that $F_n(0) = 0$ for all $n$.
 Assume that $F_n$ restricts to a functor on the idempotent complete split exact category
 $(C^{\rm{q}})^n\mc{P}$. Regarding objects of $(C^{\rm{q}})^{n+1}\mc{P}$
 as acyclic complexes in $(C^{\rm{q}})^n\mc{P}$,
 the functor $F_{n+1} = NF_n\Gamma$ restricts to a functor on
 $(C^{\rm{q}})^{n+1}\mc{P} = C^{\rm{q}}((C^{\rm{q}})^n\mc{P})$,
 by Proposition~\ref{F acyclic} $(3)$ and Lemma~\ref{iterate lemma} (2).
 The first part of the result follows by induction.\\
 For the second part, if $F_0 = F$ is of finite degree, then the same induction over $n$ shows
 that $F_n$ is of finite degree for every $n$, by Proposition~\ref{F acyclic} $(4)$.
 In particular, for each $n \ge 1$, the functor $F_n = (F_{n-1})_1$ sends bounded complexes
 to bounded complexes, that is, it restricts to a functor
 \[
  F_n \colon C^{\rm{q}}_{\rm{b}}((C^{\rm{q}})^{n-1}\mc{P})
  \rightarrow C^{\rm{q}}_{\rm{b}}((C^{\rm{q}})^{n-1}\mc{P}).
 \]
 But we can say more: considering $P_{\cd}$ in $(C^{\rm{q}}_{\rm{b}})^n\mc{P}$
 as a chain complex, each of its objects is in $(C^{\rm{q}}_{\rm{b}})^{n-1}\mc{P}$,
 i.e.\@, they are bounded. We claim that the objects of $F_n(P_{\cd}) = NF_{n-1}\Gamma(P_{\cd})$
 are also bounded. The objects of $\Gamma(P_{\cd})$ are finite direct sums of the objects
 of $P_{\cd}$. Finite sums of bounded objects are bounded, so the objects of $\Gamma(P_{\cd})$
 are bounded. Therefore, by the inductive hypothesis, the objects of $F_{n-1}\Gamma(P_{\cd})$
 are also bounded. Finally, the objects of $NF_{n-1}\Gamma(P_{\cd})$ are direct summands of the
 objects of $F_{n-1}\Gamma(P_{\cd})$
 (From Definition~\ref{moore complex}, after embedding into an abelian category),
 so they are bounded as well. Therefore $F_n$ sends bounded chain complexes of bounded objects in
 $(C^{\rm{q}}_{\rm{b}})^n\mc{P}$ to bounded chain complexes of bounded objects in
 $(C^{\rm{q}}_{\rm{b}})^n\mc{P}$. This is exactly the statement that $F_n$ restricts to a functor
 \[
  F_n\colon (C^{\rm{q}}_{\rm{b}})^n\mc{P} \rightarrow (C^{\rm{q}}_{\rm{b}})^n\mc{P},
 \]
 which was to be proved.
\end{proof}

\begin{rem}
 \label{directions}
 Throughout this chapter we work with the inductive definition of
 $(C_{\rm{b}}^{\rm{q}})^n\mc{P}$, that is
 $(C_{\rm{b}}^{\rm{q}})^n\mc{P} \defeq C_{\rm{b}}^{\rm{q}}((C_{\rm{b}}^{\rm{q}})^{n-1}\mc{P})$
 for $n > 1$. As explained in Remark \ref{commute}, one can instead think of objects in
 $(C_{\rm{b}}^{\rm{q}})^n\mc{P}$ as $\bZ_{\ge 0}^n$-graded objects of $\mc{N}$
 (together with certain differentials) without specifying the order of directions in which the
 objects have been obtained in the inductive definition. The purpose of this remark is to convince
 the reader that our construction of the functors $F_n$ given in this section
 (and hence our construction of exterior powers in the sequel) does not depend on the order of
 directions either. Rather than including a complete proof, we sketch the idea in the case $n=2$.
 Let $F_0 = F$ be as before. The functor $F_2$ is defined as
 \[
  N_hF_1\Gamma_h = N_hN_vF_0\Gamma_v\Gamma_h,
 \]
 where the indices $h$ and $v$ indicated the horizontal and vertical directions respectively.
 It is quite straightforward to see that the composition $N_hN_v$ sends a bisimplicial object $C$
 to the double complex whose objects are obtained from the corresponding objects of $C$ by dividing
 out the images of all of the horizontal and vertical degeneracy maps.
 This latter description of course does not depend on the order of $N_h$ and $N_v$.
 One can show that the same holds for $\Gamma_h$ and $\Gamma_v$ by a similar argument,
 or just by recalling that $\Gamma_h$ and $\Gamma_v$ are adjoint to $N_h$ and $N_v$ respectively.
\end{rem}

We can now describe the exterior power functors that we will use to induce operations on
higher $K$-groups. The following example is the motivation for our work so far.

\begin{mainex}
\label{main example}
 Let $\proj(R)$ be the category of finitely generated projective modules over a commutative ring $R$.
 This category is both idempotent complete and split exact.
 For each $r > 0$, the usual exterior power functor
 $\Lambda^r \colon \proj (R) \rightarrow \proj (R)$ satisfies the hypotheses of
 Corollary~\ref{operations restrict} ($\Lambda^r$ has degree~$r$).
 We therefore have induced functors
 \[
  \Lambda^r_n\colon (C^{\rm{q}}_{\rm{b}})^n\proj(R) \rightarrow
  (C^{\rm{q}}_{\rm{b}})^n\proj(R)
 \]
 for all $n \ge 0$.
\end{mainex}

In general, the complex $N\Lambda^r\Gamma(P_{\cd})$ is difficult to write down explicitly.
The paper \cite{SatK} gives an algorithm that addresses this problem.
We conclude this section by computing $N\Lambda^r\Gamma(P_{\cd})$
for a very simple choice of $P_{\cd}$.

\begin{ex}
\label{ex1}
 Let $\varphi \colon P \rightarrow Q$ be an isomorphism of invertible modules over some
 commutative ring $R$, considered as an acyclic complex concentrated in degrees $0$ and $1$:
 \[
  \begin{array}{ccccccc}
   0
   & \longrightarrow
   & P
   & \stackrel{\varphi}{\longrightarrow}
   & Q
   & \longrightarrow
   & 0, \\
   {\scriptstyle 2} & & {\scriptstyle 1} & & {\scriptstyle 0} && {\scriptstyle-1}
  \end{array}
 \]
 or
 $P \stackrel{\varphi}{\longrightarrow} Q$ for short.
 In Lemma 2.2 of \cite{KockKoszul}, the second author gives an explicit calculation of
 $N\Lambda^r\Gamma (P \stackrel{\varphi}{\longrightarrow } Q)$ in terms of higher cross-effect
 functors (in fact, he does this for more general $P, Q$ and $\varphi$).
 Specifically, in degree $n$ we have:
 \[
  N\Lambda^r\Gamma (P \stackrel{\varphi}{\longrightarrow } Q)_n =
  \operatorname{cr}_n(\Lambda^r)(P, \dots, P) \oplus \operatorname{cr}_{n+1}(\Lambda^r)(Q,P, \dots, P).
 \]
  We do not wish to expound on the theory of cross-effect functors here: the interested reader can see
  \cite{EM_cross_effects} or section 1 of \cite{KockKoszul}. Instead we merely quote the properties
  of $\operatorname{cr}_n(\Lambda^r)$ that we need.
  Firstly, $\operatorname{cr}_n(\Lambda^r) = 0$ for $n > r$, as $\Lambda^r$ is of degree $r$;
  secondly, $\operatorname{cr}_r(\Lambda^r)(P_1, \dots, P_r) = P_1 \otimes \cdots \otimes P_r$;
  thirdly, if $n < r$ and if $P_1, \dots, P_n$ are all invertible, then $\operatorname{cr}_n(\Lambda^r)(P_1, \dots, P_n) = 0$.
  From these we see that:
  \[
   N\Lambda^r\Gamma (P \stackrel{\varphi}{\longrightarrow } Q)_n =
   \begin{cases}
    P^{\otimes r} & \text{if $n = r$},\\
    Q \otimes P^{\otimes (r-1)} & \text{if $n = r-1$},\\
    0 & \text{otherwise}.
  \end{cases}
  \]
  We can also read off the differential $P^{\otimes r} \rightarrow Q \otimes P^{\otimes (r-1)}$
  from Lemma 2.2 of \textit{loc.\@ cit.\@}: it is
  $\operatorname{cr}_r(\Lambda^r)(\varphi, 1, \dots ,1) = \varphi \otimes 1 \otimes \cdots \otimes 1$.
  So $N\Lambda^r\Gamma (P \stackrel{\varphi}{\longrightarrow } Q)$ is the acyclic complex
  \[
  \begin{array}{ccccccc}
   0
   & \longrightarrow
   & P \otimes P^{\otimes (r-1)}
   & \stackrel{\varphi \otimes 1}{\longrightarrow}
   & Q \otimes P^{\otimes (r-1)}
   & \longrightarrow
   & 0. \\
   {\scriptstyle r+1} & & {\scriptstyle r} & & {\scriptstyle r-1} && {\scriptstyle r-2}
  \end{array}
 \]
 Of particular note is the special case in which $P$ and $Q$ are equal to $R$
 considered as a module over itself, and $\varphi$ is given by multiplication by some $x \in R^{\times}$.
 Then $N\Lambda^r\Gamma (R \stackrel{x}{\longrightarrow } R)$ is equal to the complex
 $(R \stackrel{x}{\longrightarrow } R)$, shifted so that it is concentrated in degrees $r$ and $r-1$.
\end{ex}

\section{Operations on binary multicomplexes}
\label{Operations on binary multicomplexes}

The goal of this section is to extend the functors $F_n$ between multicomplexes of the previous section
to functors of \textit{binary} multicomplexes.
Together with the results of the previous section, this shows that if $P_{\cd}$ is a bounded acyclic binary multicomplex, then so is $\Lambda_n^r(P_{\cd})$.

\medskip

Categories of binary complexes are not so well behaved as categories of complexes.
In particular, the category of bounded acyclic binary complexes in a split exact category is not split exact.

\begin{ex}
\label{binary not projective}
Let $P$ be an object in a split exact category $\mc{P}$.
The following diagram is an admissible epimorphism
in the category of bounded acyclic binary complexes in $\mc{P}$:
\[
 \xymatrix{
 P \binr{i_1}{i_2} \ar[d]^1 & P \oplus P \binr{p_1}{p_2}
 \ar[d]^{\Sigma} & P \ar[d] \\
 P \binr{1}{1} & P \binr{}{} & 0
 }
\]
(where $i_1$ and $i_2$ are the inclusions into the first and second summands, $p_1$ and $p_2$ are the corresponding projections,
and $\Sigma = p_1 + p_2$).
But there is no splitting $P \rightarrow P \oplus P$ that commutes with both the top and bottom differentials,
so $B^{\rm{q}}_{\rm{b}}\mc{P}$ is not split exact.
\end{ex}

This difficulty means that we cannot define
exterior powers of binary multicomplexes recursively in exactly the way we have for multicomplexes.
This problem is resolvable:
we shall show that if $P_{\cd}$ is an object of $(C^{\rm{q}}_{\rm{b}})^n\proj(R)$,
then the objects of $\Lambda^r_n(P_{\cd})$ are independent of the differentials of $P_{\cd}$. Therefore it will
make sense to define the exterior power of a binary complex by applying the exterior powers we developed
above individually to the two differentials of the binary complex. The resulting pair of complexes will have
the same objects, so we consider them as a binary complex.

\begin{lem}
\label{underlying graded}
 Let $F \colon \mc{P} \rightarrow \mc{P}$ be a covariant functor on an idempotent complete exact category.
 If $P_{\cd}$ and $Q_{\cd}$ are chain complexes with the
 same underlying graded object, then $NF\Gamma (P_{\cd})$ and $NF\Gamma (Q_{\cd})$ have the same underlying graded object.
\end{lem}

\begin{proof}
 Let $B \in \mc{P}^{\Delta^{\rm{op}}}$ be a simplicial object. The objects of the complex $N(B)$
 are given by
 \[
  N(B)_n \defeq B_n \bigg/ \left(\sum_{i=0}^n \operatorname{Im} (\sigma_i \colon B_{n-1} \rightarrow B_n) \right),
 \]
 (after embedding $\mc{P}$ in a suitable abelian category) where the $\sigma_i$ are the degeneracies of $B$.
 It is enough therefore to show that the objects and
 degeneracy maps of $F\Gamma(P_{\cd})$ do not depend upon the differential of $P_{\cd}$.
 The objects of $\Gamma(P_{\cd})$ are
 direct sums of the objects of $P_{\cd}$, indexed by the surjections out of $[n]$ in $\Delta$, and do not
 depend on the differential. The degeneracy operator $\sigma_i \colon \Gamma(P_{\cd})_{n-1} \rightarrow \Gamma(P_{\cd})_n$
 is the image of the degeneracy map $\eta_i \colon [n] \rightarrow [n-1]$ in $\Delta$.
 For any surjection $\eta \colon [n-1] \twoheadrightarrow [p]$,
 the composition $\eta \eta_i$ is also a surjection, so the monomorphism in the epi-monic factorisation of
 $\eta \eta_i$ is just the identity on $[p]$. Therefore, the degeneracy operator
 $\sigma_i$ acts on $\Gamma(P_{\cd})_{n-1}$ by sending the summand corresponding to the surjection
 $\eta$ by the identity to the summand of $\Gamma(P_{\cd})_n$ corresponding to the surjection $\eta_i\eta$.
 Thus $\sigma_i$ does not depend on the differential of $P_{\cd}$. Since the objects and degeneracies
 of $\Gamma(P_{\cd})$ only depend on the underlying graded object of $P_{\cd}$, the same is true of $F\Gamma(P_{\cd})$.
 Therefore the objects of $NF\Gamma(P_{\cd})$ only depend on the underlying graded object as well.
\end{proof}

\begin{cor}
\label{underlying equal}
  Let $n \ge 1$, and let $P_{\cd}$, $Q_{\cd}$ be objects of $(C^{\rm{q}}_{\rm{b}})^n\mc{P}$.
  If $P_{\cd}$ and $Q_{\cd}$ have the same underlying $\bZ^n$-graded object, then
  $F_n(P_{\cd})$ and $F_n(Q_{\cd})$ have the same underlying $\bZ^n$-graded object.
\end{cor}

\begin{proof}
 This is a straightforward induction on $n$.
\end{proof}

We are now ready at last to define exterior powers of acyclic binary multicomplexes. Let $P_{\cd}$
be an $n$-dimensional, bounded, acyclic binary multicomplex in $\mc{P}$, i.e.\@, an
object of $(B^{\rm{q}}_{\rm{b}})^n\mc{P}$. We view the commutativity
constraints on the differentials of $P_{\cd}$ in the same way as described in Remark
\ref{commute}: as a collection of $2^n$ objects of $(C^{\rm{q}}_{\rm{b}})^n\mc{P}$.

\begin{defn}
\label{binary functor}
 For a functor $F$ that satisfies the hypotheses of Corollary~\ref{operations restrict}, we define induced functors
  \[
  F_n\colon (B^{\rm{q}}_{\rm{b}})^n\mc{P} \longrightarrow
  (B^{\rm{q}}_{\rm{b}})^n\mc{P}
 \]
 by the following procedure.
 Let $P_{\cd}$ be an object of $(B^{\rm{q}}_{\rm{b}})^n\mc{P}$,
 viewed as a collection of $2^n$ (non-binary) multicomplexes in the manner described above.
 Since these multicomplexes all have the same underlying $\bZ^n$-graded object,
 by Corollary~\ref{underlying equal} the same is true of the $2^n$ multicomplexes obtained by applying $F_n$
 (the functor defined on $(C^{\rm{q}}_{\rm{b}})^n\mc{P}$ in Corollary~\ref{operations restrict})
 to the multicomplexes describing $P_{\cd}$.
 We define $F_n(P_{\cd})$ to be the binary multicomplex
 described by the resulting collection of multicomplexes.
\end{defn}

We now return to our main example of interest: the exterior power functors.
Let $R$ be a commutative ring. We have seen in \ref{main example} that the
usual exterior power operations $\Lambda^r$ satisfy the hypotheses of \ref{operations restrict},
so the exterior powers
 \[
  \Lambda^r_n\colon (C^{\rm{q}}_{\rm{b}})^n\proj(R) \rightarrow
  (C^{\rm{q}}_{\rm{b}})^n\proj(R)
 \]
lift to exterior powers of binary multicomplexes
 \[
  \Lambda^r_n\colon (B^{\rm{q}}_{\rm{b}})^n\proj(R) \rightarrow
  (B^{\rm{q}}_{\rm{b}})^n\proj(R)
 \]
for all $n \ge 0$ and $r \ge 1$.

\section{Simplicial tensor products}
\label{Simplicial tensor products}

In this section we develop a tensor product for multicomplexes that is compatible with the
exterior powers we have defined in the previous sections.
We show that the class of this product vanishes in the appropriate $K$-group,
which will eventually be the key to showing that exterior power operations provide
homomorphisms on higher $K$-groups.

\subsection{Constructing simplicial tensor products}
In this subsection, using the Dold--Kan correspondence again,
we construct the so-called simplicial tensor product of multi-complexes and prove that it preserves
acyclicity and boundedness of complexes.

Although we are ultimately interested in the products induced from the usual tensor products of
modules (or sheaves), it is convenient in this section to work in the rather more abstract setting
of a generic idempotent complete exact category with some form of well-behaved tensor product.

\begin{defn}
\label{tensor}
 Let $\mc{P}$ be an idempotent complete exact category.
 We say that a bi-additive bifunctor $\otimes\colon \mc{P} \times \mc{P} \rightarrow \mc{P}$
 is a \emph{tensor product} if $P \otimes -$ and $- \otimes P$ are exact functors on $\mc{P}$
 for each object $P$ of $\mc{P}$.
\end{defn}

For the rest of this section, we fix such a category $\mc{P}$ with a tensor product~$\otimes$.
The reader may wish to keep in mind the example $\mc{P} = \proj(R)$,
with the usual tensor product of $R$-modules.

\begin{defn}
\label{usual tensor}
 Let $P$ be an object of $\mc{P}$, and let $(Q_{\cd}, d_Q), (R_{\cd}, d_R)$ be chain complexes in $\mc{P}$.

 \begin{enumerate}
  \item By $P \otimes Q_{\cd}$ we mean the chain complex whose $i^{\rm{th}}$ object is $P \otimes Q_i$,
  with differential $1 \otimes d_Q$. The complex $Q_{\cd} \otimes P$ is defined analogously.
  \item By $\operatorname{Tot}(Q_{\cd} \otimes R_{\cd})$ we mean the chain complex formed by taking
  the total complex of the bicomplex whose $(i,j)^{\rm{th}}$ object is $Q_i \otimes R_j$,
  and whose differentials are $d^{\rm{ver}} = d_Q \otimes (-1)^j$ and $d^{\rm{hor}} = 1 \otimes d_R$.
  This bicomplex's $i^{\rm{th}}$ row is $Q_i \otimes R_{\cd}$
  and its $j^{\rm{th}}$ column is $Q_{\cd} \otimes R_j$
 \end{enumerate}
\end{defn}

It is clear that if $Q_{\cd}$ and $R_{\cd}$ are bounded complexes, then the products
$P \otimes Q_{\cd}$ and $\operatorname{Tot}(Q_{\cd} \otimes R_{\cd})$ are bounded as well.
We'll need a couple of properties of these products.

\begin{lem}
\label{acyclic usual tensor}
 Let $P_{\cd}$ be a chain complex in $\mc{P}$.
 \begin{enumerate}
  \item The functor
   \[
    \begin{array}{crcl}
     P_{\cd} \otimes - \colon & \mc{P} & \rightarrow & C\mc{P} \\
     & Q & \mapsto & P_{\cd} \otimes Q
    \end{array}
   \]
   is exact.
  \item If $Q_{\cd}$ is an acyclic complex in $\mc{P}$,
   then the complex $\operatorname{Tot}(P_{\cd} \otimes Q_{\cd})$ is acyclic.
 \end{enumerate}
\end{lem}

\begin{proof}
 The first part is straightforward, as each $P_i \otimes -$ is an exact functor.
 For the second part, if $Q_{\cd}$ is acyclic, then, since acyclic complexes are spliced together
 from short exact sequences, each of the complexes $P_n \otimes Q_{\cd}$ is acyclic.
 Therefore the rows of the bicomplex $P_{\cd} \otimes Q_{\cd}$ are acyclic.
 Our complexes are non-negative, so the total complex of this bicomplex is exact in an ambient
 abelian category by the acyclic assembly lemma (\cite{WeiHom}, 2.7.3).
 Since $\mc{P}$ is idempotent complete it supports long exact sequences,
 so $\operatorname{Tot}(P_{\cd} \otimes Q_{\cd})$ is acyclic in $\mc{P}$.
\end{proof}

To define the simplicial tensor product of complexes we need to go beyond regular simplicial objects.
A \emph{bisimplicial object} $B$ in $\mc{P}$ is a functor
$B \colon \Delta^{\rm{op}} \times \Delta^{\rm{op}} \rightarrow \mc{P}$. The \emph{diagonal}
of $B$ is the simplicial object defined by pre-composition with the usual diagonal functor
$\operatorname{diag} \colon \Delta^{\rm{op}} \rightarrow \Delta^{\rm{op}} \times \Delta^{\rm{op}}$:
\[
 \operatorname{diag}(B) \defeq B \circ \operatorname{diag} \colon
\Delta^{\rm{op}} \rightarrow \Delta^{\rm{op}} \times \Delta^{\rm{op}} \rightarrow \mc{P}.
\]
If $C$ and $D$ are simplicial objects in $\mc{P}$, then
we define $C \otimes D$ to be the bisimplicial object given by
$(C \otimes D)([m],[n]) = C_m \otimes D_n$ and
$(C \otimes D)(\alpha,\beta) = C(\alpha) \otimes D(\beta)$ for
$\alpha \colon [m] \rightarrow [m']$, $\beta\colon [n] \rightarrow [n']$.
We can now push the tensor product around the Dold--Kan correspondence.

\begin{defn}
\label{simplicial tensor product}
 The \emph{simplicial tensor product} of chain complexes $P_{\cd}$ and $Q_{\cd}$
 in $\mc{P}$ is defined to be
 \[
  P_{\cd} \otimes_{\Delta} Q_{\cd} \defeq N(\operatorname{diag}(\Gamma(P_{\cd}) \otimes \Gamma(Q_{\cd}))).
 \]
\end{defn}

A word of warning here: although the tensor product is an additive functor in each variable,
the complex $P_{\cd} \otimes_{\Delta} Q_{\cd}$ is not equal to the product complex
$\operatorname{Tot}(P_{\cd} \otimes Q_{\cd})$ discussed above. They are related by
the Eilenberg--Zilber theorem, which we shall use in the proof of the following lemma.

\begin{lem}
\label{comparison}
 Let $P_{\cd}$ and $Q_{\cd}$ be chain complexes in $\mc{P}$, and suppose that at least one of them is acyclic.
 Then $P_{\cd} \otimes_{\Delta} Q_{\cd}$ is acyclic in $\mc{P}$.
\end{lem}

\begin{proof}
 We suppose, without loss of generality, that $Q_{\cd}$ is acyclic.
 By the Eilenberg--Zilber theorem (\cite{May} \S 29), the simplicial tensor product
 $P_{\cd} \otimes_{\Delta} Q_{\cd} = N\operatorname{diag}(\Gamma(P_{\cd}) \otimes \Gamma(Q_{\cd}))$
 is homotopy equivalent to $\operatorname{Tot}(P_{\cd} \otimes Q_{\cd})$,
 and is therefore acyclic by Lemma~\ref{acyclic usual tensor} $(2)$ and Lemma~\ref{acyclic and split}.
\end{proof}

The following is an analogue of Lemma~\ref{length} for the simplicial tensor product.

\begin{lem}
\label{tensor bounded}
 If $P_{\cd}$ and $Q_{\cd}$ are both bounded chain complexes in $\mc{P}$, of length $k$ and $l$ respectively.
 Then $P_{\cd} \otimes_{\Delta} Q_{\cd}$ is of length at most $kl$ and so is bounded as well.
\end{lem}

\begin{proof}
 Examining the Dold--Kan functors applied to a tensor product,
 one sees that the object $(P_{\cd} \otimes_{\Delta} Q_{\cd})_n$ is equal to
 \[
  N(\operatorname{diag}(\Gamma(P_{\cd}) \otimes \Gamma(Q_{\cd})))_n =
  \bigoplus_{\varphi} P_i \otimes Q_j,
 \]
 where $\varphi$ runs over all injections $[n] \hookrightarrow [i] \times [j]$
 whose composition with the projections onto $[i]$ and $[j]$ gives surjections
 $[n] \twoheadrightarrow [i]$ and $[n] \twoheadrightarrow [j]$
 (this is derived in \cite{LawsonSimp}).
 The complexes $P_{\cd}$ and $Q_{\cd}$ are of length $k$ and $l$,
 so $P_i = 0$ and $Q_j = 0$ for all $i>k$ and $j > l$. But for $n > kl$
 there is no injection $[n] \hookrightarrow [i] \times [j]$, with $i \le k$ and $j \le l$,
 such that $[n] \twoheadrightarrow [i]$ and $[n] \twoheadrightarrow [j]$ are order-preserving
 surjections. So $(P_{\cd} \otimes_{\Delta} Q_{\cd})_n = 0$ for $n > kl$.
\end{proof}

We now verify that $\otimes_{\Delta}$ is a tensor product in the sense of Definition~\ref{tensor}.

\begin{prop}
 \label{chain complexes tensor}
 The simplicial tensor product $\otimes_{\Delta}$ is a tensor product on the idempotent complete exact category $C\mc{P}$ and restricts to a tensor product on the full subcategory $C_{\rm{b}}^{\rm{q}}\mc{P}$.
\end{prop}

\begin{proof}
 If $P_{\cd}$ and $Q_{\cd}$ are in $C_{\rm{b}}^{\rm{q}}\mc{P}$,
 then so is $P_{\cd} \otimes_{\Delta,n} Q_{\cd}$, by Lemmas \ref{comparison} and~\ref{tensor bounded}.
 So it remains to show that $- \otimes_{\Delta} -$ is bi-additive,
 and that the functors $P_{\cd} \otimes_{\Delta} - $ and $- \otimes_{\Delta} P_{\cd}$ are exact when $P_{\cd}$
 is in $C\mc{P}$.\\
 The functors $N$ and $\Gamma$ are both additive and exact,
 so we only need to inspect $\operatorname{diag}(- \otimes -)$.
 This is easily seen to be bi-additive, as $- \otimes -$ is bi-additive.
 Therefore $- \otimes_{\Delta} -$ is bi-additive as well.\\
 Let $B$ be a simplicial object in $\mc{P}$.
 For a short exact sequence of simplicial objects $0 \rightarrow A' \rightarrow A \rightarrow A'' \rightarrow 0$,
 the sequence
 \[
  0 \rightarrow \operatorname{diag}(B \otimes A')_n \rightarrow \operatorname{diag}(B \otimes A)_n \rightarrow \operatorname{diag}(B \otimes A'')_n \rightarrow 0
 \]
 is equal to
 \[
  0 \rightarrow B_n \otimes A'_n \rightarrow B_n \otimes A_n \rightarrow B_n \otimes A''_n \rightarrow 0,
 \]
 which is short exact since each
 $0 \rightarrow A'_n \rightarrow A_n \rightarrow A''_n \rightarrow 0$
 is short exact and $B_n \otimes -$ is exact.
 So the sequence
 \[
  0 \rightarrow \operatorname{diag}(B \otimes A') \rightarrow \operatorname{diag}(B \otimes A) \rightarrow \operatorname{diag}(B \otimes A'') \rightarrow 0,
 \]
 is short exact in every degree, for any simplicial object $B$ in $\mc{P}$.
 Therefore the functor
 $\operatorname{diag}(\Gamma(P_{\cd}) \otimes -) \colon \mc{P}^{\Delta^{\rm{op}}} \rightarrow \mc{P}^{\Delta^{\rm{op}}}$
 is exact. The same is true for $\operatorname{diag}(- \otimes \Gamma(P_{\cd}))$.
 It follows that $P_{\cd} \otimes_{\Delta} - $ and $- \otimes_{\Delta} P_{\cd}$ are exact functors.
\end{proof}

We are now ready to iteratively define simplicial tensor products on categories of multicomplexes.

\begin{defn}
\label{higher simplicial tensors}
 We define \emph{simplicial tensor products }
 \[
  \otimes_{\Delta,n} \colon C^n\mc{P} \times C^n\mc{P}
 \rightarrow C^n\mc{P}
 \]
 for all $n \ge 0$ recursively:

 \begin{enumerate}
  \item $\otimes_{\Delta,0} \colon \mc{P} \times \mc{P}
  \rightarrow \mc{P}$
  is the usual tensor product $\otimes$, and
  \item by regarding objects $P_{\cd}$ and $Q_{\cd}$ of $C^{n+1}\mc{P}$ as chain complexes
  in the idempotent complete exact category $C^n\mc{P}$ with the tensor product
  $\otimes_{\Delta,n}$, we define
  $P_{\cd} \otimes_{\Delta,n+1} Q_{\cd} \defeq N(\operatorname{diag}(\Gamma(P_{\cd}) \otimes_{\Delta,n} \Gamma(Q_{\cd})))$.
 \end{enumerate}
\end{defn}

 The following iteration of Proposition~\ref{chain complexes tensor} is now straightforward.
The case $n=0$ is an assumption of this section, and we iterate using
$(C_{\rm{b}}^{\rm{q}})^{n+1}\mc{P} = C_{\rm{b}}^{\rm{q}}((C_{\rm{b}}^{\rm{q}})^n\mc{P})$.

\begin{cor}
\label{iterated tesnor}
 For all $n \ge 0$, the simplicial tensor product $\otimes_{\Delta,n}$ is a tensor product
 in the sense of Definition~\ref{tensor} on $C^n\mc{P}$ and on
 $(C_{\rm{b}}^{\rm{q}})^n\mc{P}$.\qed
\end{cor}

 In fact we can say a little more than this.
The following lemma is crucial to the proof of the main result of this section.

\begin{lem}
\label{higher simplicial tensor acyclic}
 Let $P_{\cd}$ be an object of
 $C_{\rm{b}}((C_{\rm{b}}^{\rm{q}})^n\mc{P})$, and let
 $Q_{\cd}$ be an object of $(C_{\rm{b}}^{\rm{q}})^{n+1}\mc{P}$.
 Then $P_{\cd} \otimes_{\Delta, n+1} Q_{\cd}$ is an object of $(C_{\rm{b}}^{\rm{q}})^{n+1}\mc{P}$.
\end{lem}

\begin{proof}
 Noting that $P_{\cd}$ and $Q_{\cd}$ both have their objects in $(C_{\rm{b}}^{\rm{q}})^n\mc{P}$, and that $Q_{\cd}$
 is an acyclic complex of objects in that category,
 the lemma follows immediately from Lemmas \ref{comparison} and \ref{tensor bounded} applied to the tensor product
 $\otimes_{\Delta,n}$ on the category
 $(C_{\rm{b}}^{\rm{q}})^n\mc{P}$.
\end{proof}

We can extend the simplicial tensor products to categories of binary complexes in the same way
that we did for exterior powers in
section \ref{Operations on binary multicomplexes}. The simplicial tensor product
of a pair of binary complexes $(P_{\cd}, d_P, \tilde{d}_P)$ and
$(Q_{\cd}, d_Q, \tilde{d}_Q)$ is obtained by considering the pair of chain complexes
$(P_{\cd}, d_C) \otimes_{\Delta} (Q_{\cd}, d_Q)$ and
$(P_{\cd}, \tilde{d}_P) \otimes_{\Delta} (Q_{\cd}, \tilde{d}_Q)$
as a binary complex (it is straightforward to prove that they have the same underlying graded object,
in the same manner as Lemma~\ref{underlying graded}). The analogue of Corollary~\ref{underlying equal}
then follows, and we define the simplicial tensor product of binary multicomplexes just as we did
for a functor of one variable in Definition \ref{binary functor}.

\subsection{Vanishing of products}
In this subsection we prove that the class of any simplicial tensor product vanishes in the corresponding $K$-group. Our proof resembles Grayson's procedure \cite[p.~103]{Gray92} of verifying that the second Euler characteristic of a doubly acyclic bicomplex vanishes.

Let $n >0$, and let $P_{\cd}$ and $Q_{\cd}$
be $n$-dimensional bounded acyclic binary complexes of objects of $\mc{P}$.
That is, $P_{\cd}$ and $Q_{\cd}$ are objects of $(B^{\rm{q}}_{\rm{b}})^n\mc{P}$.
Then the simplicial tensor product
$P_{\cd} \otimes_{\Delta,n} Q_{\cd}$ is in $(B^{\rm{q}}_{\rm{b}})^n\mc{P}$ as well
by Corollary~\ref{iterated tesnor}. Since the objects of $(B^{\rm{q}}_{\rm{b}})^n\mc{P}$
are the generators of $K_n(\mc{P})$, one would like to use $\otimes_{\Delta,n}$ to induce a product
$K_n(\mc{P}) \times K_n(\mc{P}) \rightarrow K_n(\mc{P})$.
On first inspection this appears not to work, because
the product $P_{\cd} \otimes_{\Delta,n} Q_{\cd}$ is not diagonal if
only one of $P_{\cd}$ or $Q_{\cd}$ is diagonal. This is not a problem in the end though, since the
whole product vanishes on $K_n(\mc{P})$.

\begin{prop}
\label{tensor product vanishes}
 Let $n >0$. For any pair of $n$-dimensional bounded acyclic multicomplexes
 $P_{\cd}$, $Q_{\cd}$ in
 $(B^{\rm{q}}_{\rm{b}})^n\mc{P}$, the class
 $[ P_{\cd} \otimes_{\Delta, n} Q_{\cd} ]$ vanishes in $K_n(\mc{P})$.
\end{prop}

\begin{proof}
 First we filter $P_{\cd}$ by degree. Regard $P_{\cd}$
 as an acyclic binary complex of objects of $(B^{\rm{q}}_{\rm{b}})^{n-1}\mc{P}$.
 For $i \ge 0$, let $P|_{[0,i]}$ be the binary complex
 obtained by `restricting' $P_{\cd}$ to be supported on $[0,i]$. That is,
 $(P|_{[0,i]})_j$ is equal to $P_j$ if $0 \le j \le i$,
 and $(P|_{[0,i]})_j = 0$ otherwise. The differentials on
 $P|_{[0,i]}$ are inherited from $P_{\cd}$. We write $P_j[0]$ for
 $P_j$ considered as a binary complex concentrated in degree $0$. Then
 $P_j[j]$, which denotes $P_j$ considered as a binary complex concentrated in degree $j$,
 is the quotient of the inclusion $P|_{[0,j-1]} \hookrightarrow P|_{[0,j]}$ (if $j \ge 1$).
 If $P_{\cd}$ is supported on $[0,n]$, so that $P_j=0$ for $j>n$, we therefore have an
 $n$-stage filtration
 \[
  P_0[0] = P|_{[0,0]} \hookrightarrow P|_{[0,1]} \hookrightarrow \cdots \hookrightarrow P|_{[0,n-1]} \hookrightarrow
  P|_{[0,n]} = P_{\cd}
 \]
 whose successive quotients determine short exact sequences
 \[
  0 \rightarrow P|_{[0,j-1]} \rightarrow P|_{[0,j]} \rightarrow P_j[j] \rightarrow 0.
 \]
 We take the simplicial tensor product with $Q_{\cd}$
 of this whole filtration, obtaining sequences
 \begin{equation}
  \label{eqn:Qses}
    0 \rightarrow P|_{[0,j-1]} \otimes_{\Delta, n} Q_{\cd} \rightarrow P|_{[0,j]} \otimes_{\Delta, n} Q_{\cd}
  \rightarrow P_j[j] \otimes_{\Delta, n} Q_{\cd} \rightarrow 0
  \end{equation}
 for $j = 1, \dots , n$, which are short exact by Corollary \ref{iterated tesnor}.\\
 By Lemma \ref{higher simplicial tensor acyclic}, all of the objects are in the right category,
 so each of the short exact sequences of \eqref{eqn:Qses} yields an equation
 \[
  [ P|_{[0,j]} \otimes_{\Delta, n} Q_{\cd}] = [P|_{[0,j-1]} \otimes_{\Delta, n} Q_{\cd}]
  + [P_j[j] \otimes_{\Delta, n} Q_{\cd}]
 \]
 in $K_n(\mc{P})$. Putting these together gives
 \[
  [P_{\cd} \otimes_{\Delta, n} Q_{\cd}] = \sum_{j=0}^n [P_j[j] \otimes_{\Delta, n} Q_{\cd}].
 \]
 To proceed we need to assume a small lemma, for which the second type of relation in $K_n(\mc{P})$
 (diagonal binary multicomplexes vanish) is crucial.
 \begin{lem}
 \label{supplementary lemma}
  The following equality holds in $K_n(\mc{P})$:
  \[
   [P_j[j] \otimes_{\Delta, n} Q_{\cd}] = (-1)^j[P_j[0] \otimes_{\Delta, n} Q_{\cd}].
  \]
 \end{lem}
  Continuing with the main proof, our equation now reads
 \[
  [P_{\cd} \otimes_{\Delta, n} Q_{\cd}] = \sum_{j=0}^n (-1)^j[P_j[0] \otimes_{\Delta, n} Q_{\cd}].
 \]
 By inspection we see that $\Gamma(P_j[0])$
 is the constant simplicial object which has $P_j$
 in each degree. The functor
 \[
  \operatorname{diag}(\Gamma(P_j[0]) \otimes_{\Delta, n-1} -) \colon
 \left((C_{\rm{b}})^{n-1}\mc{A}\right)^{\Delta^{\rm{op}}} \rightarrow \left((C_{\rm{b}})^{n-1}\mc{A}\right)^{\Delta^{\rm{op}}}
 \]
 is therefore
 isomorphic to the functor
 \[
  P_j \otimes_{\Delta, n-1} - \colon
 \left((C_{\rm{b}})^{n-1}\mc{A}\right)^{\Delta^{\rm{op}}} \rightarrow \left((C_{\rm{b}})^{n-1}\mc{A}\right)^{\Delta^{\rm{op}}},
 \]
 since they both have the same effect of `tensoring everywhere by $P_j$'. This functor is additive,
 so we have an isomorphism of functors
 \[
  N( P_j \otimes_{\Delta, n-1} \Gamma(-)) \cong P_j \otimes_{\Delta, n-1} -.
 \]
 Hence
 \[
  P_j[0] \otimes_{\Delta, n} Q_{\cd} =
  N\operatorname{diag}(\Gamma(P_j[0]) \otimes_{\Delta, n-1} \Gamma(Q_{\cd})) \cong
  P_j \otimes_{\Delta, n-1} Q_{\cd},
 \]
 so we have
 \[
  [P_{\cd} \otimes_{\Delta, n} Q_{\cd}] = \sum_{j=0}^n (-1)^j[P_j \otimes_{\Delta, n-1} Q_{\cd}].
 \]
 There is an exact sequence
 \[
  0 \rightarrow P_n \rightarrow P_{n-1} \rightarrow \cdots \rightarrow P_1 \rightarrow P_0 \rightarrow 0,
 \]
 since $P_{\cd}$ is acyclic. The objects of $Q_{\cd}$ are in $(B^{\rm{q}}_{\rm{b}})^{n-1}\mc{P}$, so
 $- \otimes_{\Delta, n-1} Q_{\cd}$ is an exact functor by Lemma~\ref{acyclic usual tensor} $(1)$, and
 so the following sequence is exact.
 \begin{multline*}
  0 \rightarrow P_n \otimes_{\Delta, n-1} Q_{\cd} \rightarrow P_{n-1} \otimes_{\Delta, n-1} Q_{\cd}
  \rightarrow \cdots \\
  \cdots \rightarrow
  P_1 \otimes_{\Delta, n-1} Q_{\cd} \rightarrow P_0 \otimes_{\Delta, n-1} Q_{\cd} \rightarrow 0
 \end{multline*}
 Exact sequences translate into alternating sums in the Grothendieck group,
 so this exact sequence gives exactly the identity
 \[
  \sum_{j=0}^n (-1)^j[P_j \otimes_{\Delta, n-1} Q_{\cd}] = 0
 \]
 in $K_0\left((B_{\rm{b}}^{\rm{q}})^n\mc{P}\right)$, and hence the same relation holds in
 $K_n(\mc{P})$. Therefore $[P_{\cd} \otimes_{\Delta, n} Q_{\cd}] = 0$, as required.
\end{proof}

It remains to prove Lemma~\ref{supplementary lemma}.

\begin{proof}[Proof of Lemma~\ref{supplementary lemma}]
 Consider the following diagram as a short exact sequence of binary complexes
 concentrated in degrees $j$ and $j-1$:
 \[
  \xymatrix{
  0 \binr{}{} \ar@{ >->}[d] & P_j \ar@{ >->}[d] \\
  P_j \binr{1}{1} \ar@{->>}[d] & P_j \ar@{->>}[d] \\
  P_j \binr{}{} & 0.
  }
 \]
 We will use this diagram to show that
 $[P_j[j] \otimes_{\Delta,n} Q_{\cd}] = -[P_j[j-1] \otimes_{\Delta,n} Q_{\cd}]$.
 The argument can be iterated $j-1$ times to yield
 $[P_j[j] \otimes_{\Delta,n} Q_{\cd}] = (-1)^j[P_j[0] \otimes_{\Delta,n} Q_{\cd}]$
 in $K_n(\mc{P})$, as required.
 For lack of a better notation, we will denote the middle row of the diagram by $(P_j = P_j)$.
 Then the diagram represents a short exact sequence of binary complexes
 \[
  0 \rightarrow P_j[j-1] \rightarrow (P_j = P_j) \rightarrow P_j[j] \rightarrow 0,
 \]
 which upon tensoring with $Q_{\cd}$ becomes the short exact sequence
 \[
  0 \rightarrow P_j[j-1] \otimes_{\Delta,n} Q_{\cd} \rightarrow
  (P_j = P_j) \otimes_{\Delta,n} Q_{\cd} \rightarrow
  P_j[j] \otimes_{\Delta,n} Q_{\cd} \rightarrow 0,
 \]
 by Lemma~\ref{higher simplicial tensor acyclic}.
 Since $Q_{\cd}$ is acyclic and has objects in $(B_{\rm{b}}^{\rm{q}})^{n-1}\mc{P}$, each of the terms of this short exact sequence
 is an object of $(B_{\rm{b}}^{\rm{q}})^n\mc{P}$ by Lemma~\ref{higher simplicial tensor acyclic}, so we have a relation
 \[
  [(P_j = P_j) \otimes_{\Delta,n} Q_{\cd}] = [P_j[j-1] \otimes_{\Delta,n} Q_{\cd}]
 + [P_j[j] \otimes_{\Delta,n} Q_{\cd}],
 \]
 in $K_0\left((B_{\rm{b}}^{\rm{q}})^n\mc{P}\right)$, and hence in
 $K_n(\mc{P})$.
 We claim that
 \[
  [(P_j = P_j) \otimes_{\Delta,n} Q_{\cd}] = 0
 \]
 in $K_n(\mc{P})$, so that
 $[P_j[j] \otimes_{\Delta,n} Q_{\cd}] = -[P_j[j-1] \otimes_{\Delta,n} Q_{\cd}]$.
 We can filter $Q_{\cd}$
 in the same manner that we have filtered $P_{\cd}$ in the main proof above,
 \[
  Q_0[0] = Q|_{[0,0]} \hookrightarrow Q|_{[0,1]} \hookrightarrow \cdots \hookrightarrow Q|_{[0,n-1]} \hookrightarrow
  Q|_{[0,n]} = Q_{\cd}
 \]
 giving short exact sequences
 \[
  0 \rightarrow Q|_{[0,i-1]} \rightarrow Q|_{[0,i]} \rightarrow Q_i[i] \rightarrow 0.
 \]
 Upon tensoring with $(P_j = P_j)$, we have short exact sequences:
 \begin{multline*}
  0 \longrightarrow (P_j = P_j) \otimes_{\Delta, n} Q|_{[0,i-1]} \longrightarrow
  (P_j = P_j) \otimes_{\Delta, n} Q|_{[0,i]} \\ \longrightarrow
  (P_j = P_j) \otimes_{\Delta, n} Q_i[i] \longrightarrow 0,
 \end{multline*}
 (by Lemma \ref{higher simplicial tensor acyclic}).
 Furthermore, since $(P_j = P_j)$ is an acyclic binary complex of objects of
 $(B_{\rm{b}}^{\rm{q}})^{n-1}\mc{P}$,
 each of the terms of these short exact sequences is an object of
 $(B_{\rm{b}}^{\rm{q}})^n\mc{P}$, by Lemma \ref{higher simplicial tensor acyclic}.
 We therefore have the equation
 \[
  [(P_j = P_j) \otimes_{\Delta,n} Q_{\cd}] =
  \sum_i [(P_j = P_j) \otimes_{\Delta, n} Q_i[i]]
 \]
 in $K_0\left((B_{\rm{b}}^{\rm{q}})^n\mc{P}\right)$, and hence in
 $K_n(\mc{P})$. But $(P_j = P_j)$ is a diagonal binary complex, as is each $Q_i[i]$
 (trivially). The simplicial tensor product of a pair of diagonal complexes is again diagonal,
 so each of the acyclic binary complexes $(P_j = P_j) \otimes_{\Delta, n} Q_i[i]$ is diagonal
 and hence vanishes in $K_n(\mc{P})$. Therefore $[(P_j = P_j) \otimes_{\Delta,n} Q_{\cd}] = 0$,
 so the desired relation holds.
\end{proof}

This finally completes the proof of Proposition~\ref{tensor product vanishes}. Having taken the trouble to set
up an alternative product of bounded acyclic binary multicomplexes,
one that is compatible with the exterior powers, we've now shown that (like the usual tensor product)
it is always zero!
It was not all for naught though:
at least we know now that the induced operation
$\otimes_{\Delta,n} \colon K_n(\mc{P}) \times K_n(\mc{P}) \rightarrow K_n(\mc{P})$
is well-defined. Furthermore, the vanishing of this product proves that the exterior power operations
induce homomorphisms on $K_n(R)$ (and more generally, on the higher $K$-groups of schemes). This is shown in the next section.

\section{Exterior power operations on \texorpdfstring{$K$}{K}-groups of schemes}
\label{proof of main theorem}

The goal of this section is to extend the endofunctor
$\Lambda_n^r$ defined in section~\ref{Operations on binary multicomplexes}
to bounded acyclic multicomplexes of locally free modules of finite rank on a scheme $X$,
and to prove that it induces a well-defined operation~$\lambda^r$ on the higher $K$-group $K_n(X)$.
We will see that, for $n>0$, this operation~$\lambda^r$ is not just a map but in fact a homomorphism.

\medskip

Let $X$ be a quasi-compact scheme, and let $\proj(X)$ be the category
of locally free $\mc{O}_X$-modules of finite rank.
Then $\proj(X)$ is an exact category in the usual sense.
It is idempotent complete but not split exact in general.
We write $K_n(X)$ for the $K$-group $K_n(\proj(X))$.\\
As in section \ref{Operations on acyclic complexes}, we inductively define an endofunctor
$\Lambda^r_n$ on $C^n\proj(X)$ for $r \ge 1$ and $n \ge 0$ as follows: the functor
$\Lambda^r_0$ is the usual $r^{\rm{th}}$ exterior power functor on $C^0\proj(X) = \proj(X)$,
and $\Lambda_n^r$ is defined as $N\Lambda_{n-1}^r\Gamma$, with $N$ and $\Gamma$ as introduced
in section \ref{prelim}.

\begin{prop}
\label{schemes restrict}
 For all $r,n>0$, the functor $\Lambda_n^r$ restricts to an endofunctor on the subcategory
 $(C^{\rm{q}}_{\rm{b}})^n\proj(X)$ of $C^n\proj(X)$.
\end{prop}

\begin{proof}
 Given any open affine subscheme $U = \operatorname{Spec}(R)$ of $X$, a straightforward inductive argument shows that
 the following diagram commutes:
 \[
  \xymatrix{
  C^n\proj(X) \ar[r]^{\Lambda_n^r} \ar[d] & C^n\proj(X) \ar[d] \\
  C^n\proj(U) \ar[r]^{\Lambda_n^r} \ar@{=}[d] & C^n\proj(U) \ar@{=}[d] \\
  C^n\proj(R) \ar[r]^{\Lambda_n^r} & C^n\proj(R)
  }
 \]
 The vertical arrows are induced by the restriction functor
 $\proj(X) \rightarrow \proj(U)$, $P \mapsto P|_U$,
 and the lower horizontal arrow is the functor $\Lambda_n^r$ introduced in section
 \ref{Operations on acyclic complexes}. A complex in $C^n\proj(X)$ is
 acyclic, or bounded, if and only if its restriction to every open affine subscheme
 has the respective property,
 so Proposition~\ref{schemes restrict} follows from the results of section~\ref{Operations on acyclic complexes}.
\end{proof}

As in section \ref{Operations on binary multicomplexes}, one easily deduces that, for any complex
$P_{\cd}$ in $C^n\proj(X)$, the objects in $\Lambda_n^r(P_{\cd})$ do not
depend on the differentials in $P_{\cd}$. We can therefore extend the endofunctor
$\Lambda_n^r$ to an endofunctor of $(B^{\rm{q}}_{\rm{b}})^n\proj(X)$, which we denote by
$\Lambda_n^r$ again. The goal of the rest of this section is to prove the
following theorem.

\begin{thm}
 \label{main theorem for schemes}
 Let $n > 0$ and $r > 0$. The endofunctor $\Lambda_n^r$ of $(B^{\rm{q}}_{\rm{b}})^n\proj(X)$
 induces a well-defined homomorphism $\lambda^r \colon K_n(X) \rightarrow K_n(X)$.
\end{thm}

\begin{defn}
 The homomorphism $\lambda^r$ in the previous theorem is called the
 \emph{$r^{\rm{th}}$ exterior power operation} on $K_n(X)$.
\end{defn}

\begin{proof}[Proof of Theorem~\ref{main theorem for schemes}]
If $P_{\cd}$ is a diagonal multicomplex in $(B^{\rm{q}}_{\rm{b}})^n\proj(X)$,
then the multicomplex $\Lambda_n^r(P_{\cd})$ is diagonal as well, by definition of
$\Lambda^r_n$. It therefore suffices to show that the association
$[P_{\cd}] \mapsto [\Lambda_n^r(P_{\cd})]$ induces a well-defined homomorphism
of groups
\[
 \lambda^r \colon K_0\left((B^{\rm{q}}_{\rm{b}})^n\proj(X)\right) \rightarrow K_n(X).
\]
Thus we need to show that the equality
\[
 [\Lambda_n^r(P_{\cd})] = [\Lambda_n^r(P'_{\cd})] + [\Lambda_n^r(P''_{\cd})]
\]
holds in $K_n(X)$ for every short exact sequence
$0 \rightarrow P'_{\cd} \rightarrow P_{\cd} \rightarrow P''_{\cd} \rightarrow 0$
in $(B^{\rm{q}}_{\rm{b}})^n\proj(X)$.
The classes $[\Lambda_n^{r-i}(P'_{\cd}) \otimes_{\Delta, n} \Lambda_n^i(P''_{\cd})],$
$i = 1, \dots , r-1$, vanish in $K_n(X)$ by Proposition~\ref{tensor product vanishes} applied
to the category $\mc{P} = \proj(X)$, where the simplicial tensor product has been
constructed inductively from the usual tensor product of quasi-coherent $\mc{O}_X$-modules.
So the desired equality is equivalent in $K_n(X)$ to the more familiar-looking identity
\[
 [\Lambda_n^r(P_{\cd})]
 = [\Lambda_n^r(P'_{\cd})]
 + \sum_{i=1}^{r-1}[\Lambda_n^{r-i}(P'_{\cd}) \otimes_{\Delta, n} \Lambda_n^i(P''_{\cd})]
 + [\Lambda_n^r(P''_{\cd})].
\]
In order to prove this latter formula, we cannot just apply the usual formula
for the $r^{\rm{th}}$ exterior power of a direct sum because the given short exact sequence of
binary complexes,
$0 \rightarrow P'_{\cd} \rightarrow P_{\cd} \rightarrow P''_{\cd} \rightarrow 0$,
does not split in general, even if $X$ is affine
(see Example~\ref{binary not projective}).
Instead, by induction on $n$, we construct for every sequence
$0 \rightarrow P'_{\cd} \rightarrow P_{\cd} \rightarrow P''_{\cd} \rightarrow 0$
in $(B^{\rm{q}}_{\rm{b}})^n\proj(X)$ a natural induced filtration
\[
 \Lambda_n^r(P'_{\cd}) \hookrightarrow
 \Lambda_n^{r-1}(P'_{\cd}) \wedge_n \Lambda_n^1(P_{\cd}) \hookrightarrow
 \cdots \hookrightarrow
 \Lambda_n^1(P'_{\cd}) \wedge_n \Lambda_n^{r-1}(P_{\cd}) \hookrightarrow
 \Lambda_n^r(P_{\cd})
\]
of $\Lambda_n^r(P_{\cd})$ by certain sub-objects
$\Lambda_n^{r-i}(P'_{\cd}) \wedge_n \Lambda_n^i(P_{\cd})$, $i = 0, \dots, r$
of $\Lambda_n^r(P_{\cd})$, also belonging to $(B^{\rm{q}}_{\rm{b}})^n\proj(X)$,
together with short exact sequences
\begin{multline}
  \label{eqn:wedge quotients}
    0 \longrightarrow
    \Lambda_n^{r-i+1}(P'_{\cd}) \wedge_n \Lambda_n^{i-1}(P_{\cd}) \longrightarrow
    \Lambda_n^{r-i}(P'_{\cd}) \wedge_n \Lambda_n^i(P_{\cd}) \\ \longrightarrow
    \Lambda_n^{r-i}(P'_{\cd}) \otimes_{\Delta, n} \Lambda_n^i(P''_{\cd}) \longrightarrow
    0,
  \end{multline}
  $i = 1, \dots , n$.\\
  For $n= 0$ and $i \in \{ 0 , \dots , r \}$, the object
  $\Lambda_0^{r-i}(P') \wedge_0 \Lambda_0^i(P)$ is defined to be what is usually meant by
  $\Lambda^{r-i}(P') \wedge \Lambda^i(P)$: the image of the canonical homomorphism
  $\Lambda^{r-i}(P') \otimes \Lambda^i(P) \rightarrow \Lambda^r(P)$. It is well-known
  that these objects come with the required short exact sequences \eqref{eqn:wedge quotients}.\\
  If $n > 0$ and if, for a moment, the sequence
  $0 \rightarrow P'_{\cd} \rightarrow P_{\cd} \rightarrow P''_{\cd} \rightarrow 0$
  is given in $C^{\rm{q}}_{\rm{b}}(C^{\rm{q}}_{\rm{b}})^{n-1}\proj(X)$ rather than in $(B^{\rm{q}}_{\rm{b}})^n\proj(X)$,
  we first note that applying the exact functor $\Gamma$ to the sequence, we get the short exact sequence
  $0 \rightarrow \Gamma(P'_{\cd}) \rightarrow \Gamma(P_{\cd}) \rightarrow
  \Gamma(P''_{\cd}) \rightarrow 0$
  of simplicial objects in $(C^{\rm{q}}_{\rm{b}})^{n-1}\proj(X)$. By the inductive hypothesis, the complexes
  $\Lambda_{n-1}^{r-i}(\Gamma(P')_m) \wedge_{n-1} \Lambda_{n-1}^i(\Gamma(P)_m),$
  $i = 0, \dots , n$, $m \ge 0$, belong to $(B^{\rm{q}}_{\rm{b}})^{n-1}\proj(X)$ and we have short exact sequences
  \begin{multline*}
   0 \longrightarrow
   \Lambda_{n-1}^{r-i+1}(\Gamma(P')_m) \wedge_{n-1} \Lambda_{n-1}^{i-1}(\Gamma(P)_m) \\ \longrightarrow
   \Lambda_{n-1}^{r-i}(\Gamma(P')_m) \wedge_{n-1} \Lambda_{n-1}^i(\Gamma(P)_m) \\ \longrightarrow
   \Lambda_{n-1}^{r-i}(\Gamma(P')_m) \otimes_{\Delta, n-1} \Lambda_{n-1}^i(\Gamma(P'')_m) \longrightarrow
   0,
  \end{multline*}
  $i = 1, \dots , r$, $m \ge 0$. These short exact sequences assemble to short exact sequences of simplicial objects in
  $(B^{\rm{q}}_{\rm{b}})^{n-1}\proj(X)$. By applying the exact functor $N$ we finally obtain the required objects
  \[
   \Lambda_n^{r-i}(P') \wedge_n \Lambda_n^i(P) \defeq
   N( \Lambda_{n-1}^{r-i}(\Gamma(P')) \wedge_{n-1} \Lambda_{n-1}^i(\Gamma(P))),
  \]
  $i = 0, \dots , r$, and the required short exact sequences \eqref{eqn:wedge quotients}.
  As the objects of the multicomplex $\Lambda_n^{r-i}(P') \wedge_n \Lambda_n^i(P)$ are independent
  of the differentials in the multicomplexes $P'_{\cd}$ and $P_{\cd}$, this
  construction of $\wedge_n$ passes to the category $(B^{\rm{q}}_{\rm{b}})^n\proj(X)$ as in section
  \ref{Operations on binary multicomplexes}.\\
  From Proposition~\ref{schemes restrict} and section \ref{Simplicial tensor products} we know that the complex
   $\Lambda^r_n(P_{\cd})$ and the complexes
   $\Lambda_n^{r-i}(P'_{\cd}) \otimes_{\Delta,n} \Lambda_n^i(P''_{\cd})$,
   $i = 0, \dots , r$, belong to $(B^{\rm{q}}_{\rm{b}})^n\proj(X)$. Now a straightforward downwards induction
   on $i$ based on the short exact sequences \eqref{eqn:wedge quotients} shows that the complexes
   $\Lambda_n^{r-i}(P'_{\cd}) \wedge_n \Lambda_n^i(P_{\cd})$,
   $i = 0, \dots , r$ are bounded and acyclic,
   so they belong to $(B^{\rm{q}}_{\rm{b}})^n\proj(X)$ as was to be shown.
\end{proof}

\section{The second \texorpdfstring{$\lambda$}{lambda}-ring axiom}
\label{Lambda axiom 4}

Given a scheme $X$, there is a `trivial' way to equip the graded abelian group
$K_{*}(X) \defeq \bigoplus_{n \ge 0} K_n(X)$
with a multiplication,
and to extend the exterior power operations defined in the previous section to
$K_{*}(X)$ so that they are compatible with addition in $K_{*}(X)$ in the usual sense.
The main result of this section is that they are also compatible with multiplication in the expected
way---that is to say, they satisfy the $\lambda$-ring axiom (2).

\medskip

Let $X$ be a quasi-compact scheme. We recall that $K_0(X)$ together with the usual exterior power operations
$\lambda^r \colon K_0(X) \rightarrow K_0(X)$, $r \ge 0$, is a $\lambda$-ring as defined
in the introduction (see chapter V of \cite{FultonLang}).
Furthermore, $K_n(X)$ is a $K_0(X)$-module via $[P] \cdot [Q_{\cd}] \defeq [P \otimes Q_{\cd}]$,
for $P$ in $\proj(X)$ and $Q_{\cd}$ in $(B_{\rm{b}}^{\rm{q}})^n\proj(X)$, see also Definition~\ref{usual tensor} (1).\\
We define a multiplication on $K_{*}(X) \defeq \bigoplus_{n \ge 0} K_n(X)$ by
\[
 (a_0,a_1,a_2, \dots ) \cd (b_0,b_1,b_2,\dots) = (a_0b_0, a_0b_1 + a_1b_0, a_0b_2 + a_2b_0, \dots);
\]
in particular, the product of any two elements in $\bigoplus_{n \ge 1} K_n(X)$ is defined to be zero.
With this multiplication, $K_{*}(X)$ is a commutative ring.
Furthermore, we define exterior power operations $\lambda^r \colon K_{*}(X) \rightarrow K_{*}(X)$, $r \ge 0$, by the formula
\[
 \lambda^r((a_0,a_1,a_2,\dots)) = \left(\lambda^r(a_0), \sum_{i=0}^{r-1}\lambda^i(a_0)\lambda^{r-i}(a_1), \sum_{i=0}^{r-1}\lambda^i(a_0)\lambda^{r-i}(a_2), \dots \right).
\]
By definition, we then have $\lambda^0(x) = 1$ and $\lambda^1(x) = x$ for all $x \in K_{*}(X)$.
A straightforward calculation using Theorem~\ref{main theorem for schemes} and the fact that $K_0(X)$ satisfies axiom $(1)$ of a $\lambda$-ring shows that
$K_{*}(X)$ also satisfies axiom $(1)$. The next theorem addresses axiom $(2)$.

\begin{thm}
 \label{lambda axiom 2}
 The pre-$\lambda$-ring $K_{*}(X)$ defined above satisfies axiom $(2)$ of a $\lambda$-ring.
\end{thm}

\begin{proof}
 Axiom $(2)$ holds for elements of the form $x = (a_0,0,0,\dots)$, $y = (b_0,0,0,\dots)$ in $K_{*}(X)$ because it holds for $K_0(X)$.
 It also holds for elements of the form $x = (0, a_1,a_2,\dots)$ and $y=(0,b_1,b_2,\dots)$ because $\lambda^r(0)=0$ for all $r \ge 1$ and because
 every monomial in the ring $\bZ[X_1,\dots,X_r,Y_1,\dots,Y_r]$ whose coefficient in $P_r(X_1,\dots,X_r,Y_1,\dots,Y_r)$ is non-zero
 is divisible by some product $X_iY_j$. Furthermore, it suffices to check axiom $(2)$ for $x,y$ belonging to a set of additive generators
 of $K_{*}(X)$ because $K_{*}(X)$ satisfies axiom $(1)$ and because axiom $(2)$ is
 equivalent to the multiplicativity of the homomorphism
 \[
  \begin{array}{rrcl}
   \lambda_t \colon & K_{*}(X) & \rightarrow & 1 + t.K_{*}(X)[[t]] \\
   & x & \mapsto & \sum_{r \ge 0}\lambda^r(x)t^r.
  \end{array}
 \]
  We are therefore reduced to showing that the equality
  \begin{equation}
  \label{eqn:lambda2}
    \lambda^r(xy) = P_r(\lambda^1(x),\dots,\lambda^r(x),\lambda^1(y),\dots,\lambda^r(y))
  \end{equation}
  holds in $K_n(X)$ for elements $y \in K_n(X)$ and $x \in K_0(X)$ of the form $x = [\ms{E}]$ for some locally free
  $\mc{O}_X$-module $\ms{E}$ of finite rank.\\
  We now invoke the projective bundle theorem (\cite{Quil1}, \S 8, Theorem 2.1).
  We remark that its proof in \textit{loc.\@ cit.\@} only relies on the additivity and resolution theorems,
  and not, for instance, on the d\'{e}vissage theorem or localisation sequence. The additivity and resolution theorems
  have been proved in \cite{Har15} within the context of Grayson's definition of higher $K$-groups, so the projective bundle theorem also
  has a proof within that context, without resorting to topological methods.\\
  It is well-known that an iterated application of the projective bundle theorem yields the following splitting principle:
  there exists a projective morphism $f \colon Y \rightarrow X$ such that $f^{*}[\ms{E}]$ is the sum of invertible $\mc{O}_Y$-modules in $K_0(Y)$
  and such that $f^{*} \colon K_{*}(X) \rightarrow K_{*}(Y)$ is injective. It is straightforward to check that
  $f^{*} \colon K_{*}(X) \rightarrow K_{*}(Y)$ is a homomorphism of (pre-)$\lambda$-rings.
  Using the above argument about additive generators again, we are therefore reduced to showing the equality \eqref{eqn:lambda2}
  only when $x$ is the class $[\ms{L}]$ of an invertible $\mc{O}_X$-module $\ms{L}$.
  In that case, \eqref{eqn:lambda2} becomes the much simpler formula
  \[
   \lambda^r([\ms{L}]\cdot y) = [\ms{L}^{\otimes r}] \cdot \lambda^r(y),
  \]
  because $\lambda^2[\ms{L}] = \dots = \lambda^r[\ms{L}] = 0$, and because $P_r$
  satisfies the identity $P_r(1,0, \dots,0,Y_1,\dots,Y_r) = Y_r$
  and has $X$-degree $r$ (where $X_i$ is defined to be of degree $i$ for $i=1,\dots,r$).
  Using the argument about additive generators again, it suffices to show that
  for any object $P_{\cd}$ of $(B_{\rm{b}}^{\rm{q}})^n\proj(X)$, the object
  $\Lambda_n^r(\ms{L} \otimes P_{\cd})$ is isomorphic to $\ms{L}^{\otimes r} \otimes \Lambda_n^r(P_{\cd})$.
  This is well-known if $n=0$, and follows by induction on $n$ from the following chain of isomorphisms applied to each of the $2^n$ multicomplexes associated
  with the binary multicomplex~$P_{\cd}$ (which we again denote by $P_{\cd}$):
  \[
   \begin{array}[b]{rcl}
    \Lambda_n^r(\ms{L} \otimes P_{\cd}) & = & N\Lambda_{n-1}^r\Gamma(\ms{L} \otimes P_{\cd}) \\
    & \cong & N (\ms{L}^{\otimes r} \otimes \Lambda_{n-1}^r\Gamma(P_{\cd})) \\
    & \cong & \ms{L}^{\otimes r} \otimes N\Lambda_{n-1}^r\Gamma(P_{\cd}) \\
    & \cong & \ms{L}^{\otimes r} \otimes \Lambda_n^r(P_{\cd}).
   \end{array}\qedhere
  \]
\end{proof}

\section{The final \texorpdfstring{$\lambda$}{lambda}-ring axiom}
\label{Lambda final}

The goal of this section is to prove that the pre-$\lambda$-ring $K_*(X)$
(introduced and proven to satisfy $\lambda$-ring axiom (2) in the previous section)
also satisfies the final $\lambda$-ring axiom (3) and is therefore a $\lambda$-ring.
The main ingredients are the language of polynomial functors,
the identification of polynomial functors with modules over the Schur algebra,
and Serre's method of computing the Grothendieck group of representations of the
group scheme~$\textrm{GL}_{n, \bZ}$.

\subsection{Polynomial functors}

In this subsection we introduce the notion of polynomial functors and state that the Grothendieck
group of the category of polynomial functors over $\mathbb{Z}$ is isomorphic to the universal
$\lambda$-ring in one variable, see Theorem~\ref{Computing Grothendieck group of PolFun} below.
This theorem will allow us to prove the final $\lambda$-ring axiom for $K_*(X)$ in
subsection~\ref{Proof of final axiom}.
The proof of Theorem~\ref{Computing Grothendieck group of PolFun} will occupy the next two subsections.

We recall $\proj(S)$ denotes the category of $\cO_S$-modules that are locally free of finite rank on a scheme $S$.
We define a category $\uP(S)$ `enriched in schemes over $S$' as follows.
The objects are the same as the objects of $\proj(S)$, and for every $V,W \in \proj(S)$ we have an $S$-scheme
\[
	\uHom(V,W) := \uSpec_S \Sym^\bullet (\cHom(V,W)^\vee).
\]
This is the `physical vector bundle' corresponding to the locally free $\cO_S$-module $\cHom(V,W)$ and we have
\[
	\uHom(V,W)(T) = \Hom_{\cO_T}( V_T, W_T )
\]
for every $S$-scheme $T$.
In fact, by Yoneda's Lemma, we may think of $\uHom(V,W)$ as the functor which associates $\Hom_{\cO_T}( V_T, W_T )$ with every $S$-scheme $T$.
The latter viewpoint is used in a lot of literature about polynomial functors.
Composition in $\uP(S)$ is given by the natural maps
\[
	\uHom(U,V) \times_S \uHom(V,W) \to \uHom(U,W)
\]
of schemes over $S$, and the identities are given by the the obvious sections $\id_V\in \uHom(V,V)(S)$.

\begin{defn}\label{PolynomialFunctor}
A \emph{polynomial functor $F$ over $S$} is an enriched functor $F\colon \uP(S) \to \uP(S)$.
A \emph{morphism of polynomial functors} is a natural transformation. We denote the category of polynomial functors over $S$ by $\Pol(S)$.
\end{defn}

In other words, a polynomial functor consists of objects $FV\in \proj(S)$,  $V\in \proj(S)$, and of morphisms of $S$-schemes
\[
	F \colon \uHom(V,W) \to \uHom(FV,FW), \quad  V,W \in \proj(S)
\]
which satisfy the usual functor axioms. In less precise terms, $Ff$ being a morphism of $S$-schemes means that, if for instance $S= \textrm{Spec}(k)$, $k$ a field, the map $Ff: \Hom(V,W) \rightarrow \Hom(FV, FW)$ is given by polynomials in coordinates of $V$ and $W$. Note that we do not ask $Ff$ to be additive. Every polynomial functor $F$ induces an `ordinary' endofunctor of $\proj(S)$, denoted by $F$ again. A morphism $\eta\colon F\to G$ consists of a morphism of $\cO_S$-modules
\[
	\eta_V \colon FV \to GV
\]
for every $V\in \proj(S)$, satisfying the usual conditions for a natural transformation.

\begin{ex}[Exterior powers] Functoriality of $\Lambda^d$ implies that for all $V,W \in \proj(S)$ we have a map
\[
	\Hom(V,W) \longto \Hom(\Lambda^d V, \Lambda^d W).
\]
This is {\em a priori} a map of sets, but its formation commutes with base change $T\to S$, and hence by Yoneda it defines a map of $S$-schemes
\[
	\uHom(V,W) \longto \uHom(\Lambda^d V, \Lambda^d W).
\]
We obtain a polynomial functor $\Lambda^d\colon \uP(S) \to \uP(S)$.
\end{ex}

The category $\Pol(S)$ is a $\Gamma(S,\cO_S)$-linear category. We declare a sequence $0\to F \to G\to H \to 0$ in $\Pol(S)$ to be exact if the sequence
\[
	0 \to FV \to GV \to HV \to 0
\]
is exact for every $V$;  this way $\Pol(S)$ becomes an exact category \cite[Section~2.1.1]{To13}. It carries a tensor product
\[
	\otimes \colon \Pol(S) \times \Pol(S) \to \Pol(S)
\]
as well as exterior power operators
\[
	\Lambda^n \colon \Pol(S) \to \Pol(S),\quad F \mapsto \Lambda^n F:=\Lambda^n\circ F.
\]
These data turn $K_0(\Pol(S))$ into a pre-$\lambda$-ring.
To prove this, one proceeds as in the proof of Theorem~\ref{main theorem for schemes}.
As there, the category $\Pol(S)$ is in general not split exact, but for every short exact sequence as above, one can construct a natural filtration
\[0 \subset \Lambda^n F \subset F \wedge \ldots \wedge F \wedge G \subset \ldots \subset F \wedge G \wedge \ldots \wedge G \subset \Lambda^n G\]
of $\Lambda^n G$ whose successive quotients are isomorphic to $\Lambda^{n-k} F \otimes \Lambda^k H$, $k=0, \ldots, n$.

Less evident is that for every morphism $f\colon T\to S$ there is a natural \emph{base change functor} $f^\ast: \Pol(S) \to \Pol(T)$.
This can be constructed as follows.
Let $F\colon \uP(S) \to \uP(S)$ be a polynomial functor.
Given $V \in \proj(T)$ one chooses an open cover $(U_i)$ of $T$,
vector bundles $V_i \in \uP(S)$ and isomorphisms $\alpha_i \colon (f^\ast V_i)|_{U_i} \to V|_{U_i}$.
These define glueing data $\alpha_{ij} := \alpha_i^{-1}\alpha_j$
and one constructs the desired $(f^\ast F)V$ by glueing the bundles $f^\ast (FV_i)$ over the $U_{ij}$ using the maps $F(\alpha_{ij})$.
Note that the expression $F(\alpha_{ij})$ makes sense as $F$ is a polynomial functor.
For an alternative description of $f^*$, see Remark~\ref{rmk:base-change}.

Thus, every polynomial functor $F\in \Pol(S)$ induces a family of functors
\[
	F_T \colon \proj(T) \to \proj(T),
\]
indexed by $T\to S$, and that the $F_T$ commute with base change.

The functor $f^\ast$ is exact, and commutes with the operations $\otimes$ and $\Lambda^n$, so that $f^\ast$ induces a morphism
\[
	f^\ast \colon K_0(\Pol(S)) \to K_0(\Pol(T))
\]
of pre-$\lambda$-rings.

\begin{defn}
A polynomial functor $F\in \Pol(S)$ is said to be \emph{homogeneous of degree $d$} if, for every $V\in \proj(S)$, the diagram
\[
\begin{tikzcd}
\bG_{m,S} \arrow{r}{x\;\mapsto\; x^d} \arrow{d} & \bG_{m,S} \arrow{d} \\
\uHom(V,V) \arrow{r}{F} & \uHom(FV,FV)
\end{tikzcd}
\]
commutes; here, the vertical morphisms are given by scalar multiplication. We denote by $\Pol_d(S)$ the category of polynomial functors homogenous of degree $d$, and by $\Pol_{<\infty}(S)$ the category of polynomial functors that are finite direct sums of homogeneous polynomial functors.
\end{defn}

\begin{ex} The polynomial functor $\Lambda^d$ is homogenous of degree $d$. The infinite direct sum $\oplus_{d\geq 0} \Lambda^d$ is well-defined as it becomes finite when applied to any $V$; it is a polynomial functor, but not in $\Pol_{<\infty}(S)$.
\end{ex}

Let $\bZ[s_1,s_2,\ldots ]$ denote the ring of symmetric functions, with $s_i$ the $i^\textrm{th}$ elementary symmetric function. This is a $\lambda$-ring, with $\lambda^i(s_1) = s_i$, also called the universal $\lambda$-ring in one variable, see \cite[\S 1.3]{Yau}. It is also a graded ring with $\deg s_d = d$.

\begin{thm}
\label{Computing Grothendieck group of PolFun}
The ring homomorphism
\[
	\bZ[s_1,s_2,\ldots ] \to K_0(\Pol_{< \infty}(\bZ)),\quad s_i \mapsto [\Lambda^i]
\]
is an isomorphism of pre-$\lambda$-rings.
\end{thm}

The proof of this theorem will be given at the end of subsection~\ref{Grothendieck groups over Z}.

\begin{cor}
\label{PolynomialFunctorsFormLambdaRing}
$K_0(\mathrm{Pol}_{< \infty}(\bZ))$ is a $\lambda$-ring. \hspace*{\fill}$\Box$
\end{cor}

\subsection{The Schur algebra}\label{SchurAlgebra}

The object of this subsection is to relate polynomial functors to the Schur algebra, see
\cite[Ch.~I \& Ch.~IV]{Roby} and \cite[\S 2]{Krause} for details.

Throughout this subsection, $R$ is a commutative ring. If $M$ is a locally free $R$-module and $d$ a non-negative integer, then the $R$-module of degree~$d$ divided powers is the module of symmetric degree~$d$ tensors:
\[\Gamma^d M = \Gamma^d_R M = (M^{\otimes d})^{S_d}.\]

If $A$ is an associative and locally free $R$-algebra and $M$ is moreover an $A$-module, then $\Gamma^d_R A$ is a sub-$R$-algebra of $A^{\otimes d}$ and the obvious multiplication of $\Gamma^d_R A$ on $\Gamma^d_R M$ turns $\Gamma^d_R M $ into a $\Gamma^d_R A$-module.

Let $n$ be a positive integer. Consider the \emph{Schur algebra} $\Gamma^d \Mat(n,R)$ of $R$ associated with $n$ and $d$. It is free as an $R$-module. For every $R$-module $V$,
the module $V^n = \Hom(R^n,V)$ is a right $\Mat(n,R)$-module,
hence $\Gamma^d(V^n)$ is a right $\Gamma^d \Mat(n,R)$-module.

\begin{lem}
\label{lemma:projective-schur-module}
If $V$ is a projective $R$-module, then $\Gamma^d(V^n)$ is a projective right $\Gamma^d \Mat(n,R)$-module.
\end{lem}

\begin{proof}
If $V$ is a direct summand of $W$, then $\Gamma^d(V^n)$ is a direct summand of $\Gamma^d(W^n)$, so without loss of generality we may assume that $V$ is a free $R$-module. Then $\Gamma^d (V^n)$ is a direct sum of $\Gamma^d\Mat(n,R)$-modules of the form
\[
	\Gamma^{d_1}(R^n) \otimes_R \cdots \otimes_R \Gamma^{d_i} (R^n)
\]
with $\sum d_i = d$. By  \cite[Prop.~2.1]{AkinBuchsbaum} these are projective over the Schur algebra $\Gamma^d\Mat(n,R)$, and the lemma follows.
\end{proof}

We denote by $\cM(R,n,d)$ the category of finitely generated left modules over the Schur algebra $\Gamma^d\textrm{Mat}(n,R)$, and by $\cM_\textrm{p}(R,n,d)$ the full sub-category consisting of those modules whose underlying $R$-module is projective.

We have a `truncation' functor,
\[
	\Pol_d(R) \to \cM_\textrm{p}(R,n,d),\quad F \mapsto F(R^n),
\]
where the structure of left $\Gamma^d\Mat(n,R)$-module on $F(R^n)$ is defined as follows.
We have a map
\[
	\uEnd(R^n) \overset{F}{\longto} \uEnd(F(R^n))
\]
which is homogeneous of degree $d$. By the universal property of divided powers (see \cite[Proposition~IV.1]{Roby} and \cite[Proposition~2.5.1]{Ferrand}), this map is induced by an $R$-module homomorphism
\[
	\Gamma^d \End(R^n) \to \End(F(R^n)),
\]
which is moreover multiplicative, hence giving $F(R^n)$ the structure of a $\Gamma^d\End(R^n)$-module.

\begin{thm}
\label{PolynomialSchur}
If $n\geq d$, the functor $\Pol_d(R) \to \cM_{\mathrm{p}}(R,n,d)$ is an equivalence of categories.
\end{thm}

\begin{proof} See \cite{Krause}, where the same result is shown for polynomial functors taking values in arbitrary $R$-modules, and arbitrary $\Gamma^d\Mat(n,R)$-modules. The same argument works in our context, we only need to check that the inverse functor maps $\cM_\mathrm{p}(R,n,d)$ to $\Pol_d(R)$ (that is, that the inverse functor preserves `finite type and projective').

The inverse functor is defined as follows. Let $M$ be a $\Gamma^d\Mat(n,R)$-module. Then we define a functor
\[
	F_M\colon \Mod(R) \to \Mod(R),\, V \mapsto \Gamma^d(V^n) \otimes_{\Gamma^d\Mat(n,R)} M,
\]
where the right $\Gamma^d\Mat(n,R)$-module structure on $\Gamma^d(V^n)$ is inherited from the structure of right $\Mat(n,R)$-module on $V^n = \Hom(R^n,V)$. Formation of $F_M$ commutes with base change.

Now assume that both $M$ and $V$ are finitely generated and projective $R$-modules. Then the module $F_M(V)$ is also finitely generated. We claim that $F_M(V)$ is also projective. By Lemma~\ref{lemma:projective-schur-module} the module $\Gamma^d(V^n)$ is projective, hence a direct summand of a free $\Gamma^d\Mat(n,R)$-module $\oplus_I \Gamma^d\Mat(n,R)$, and hence $F_M(V)$ is a direct summand of a projective $R$-module $\oplus_I M$.
\end{proof}

\begin{rem}
\label{rmk:base-change}
Theorem~\ref{PolynomialSchur} gives an alternative way for producing the base change of a polynomial functor. If $R\to S$ is a map of commutative rings, and if $M$ is a $\Gamma^d\Mat(n,R)$-module, then the base change $M\otimes_R S$ is a $\Gamma^d\Mat(n,S)$-module,
since formation of $\Gamma^d\Mat(n,-)$ commutes with base change.
\end{rem}

\subsection{The Grothendieck group of polynomial functors over \texorpdfstring{$\mb{Z}$}{Z}\ }\label{Grothendieck groups over Z}

We fix $n$ and $d$ satisfying $n\geq d$. For brevity we write $\cM(R) := \cM(R,n,d)$ and $\cM_\mathrm{p}(R) := \cM_\mathrm{p}(R,n,d)$.
Furthermore we write $\bZ[s_1, s_2, \ldots]_d$ for the weighted degree $d$ part of the polynomial ring $\bZ[s_1,s_2,\ldots]$. It is equal to $\bZ[s_1, \ldots, s_n]_d$.

In this subsection, following \cite{Serre}, we compute the Grothendieck group $K_0(\cM_\mathrm{p}(\bZ))$. Together with Theorem~\ref{PolynomialSchur} this then implies Theorem~\ref{Computing Grothendieck group of PolFun}.

If $R$ is an integral domain, there is a natural homomorphism
\[
	K_0(\cM_\mathrm{p}(R))\;\;\stackrel{\textrm{Thm.~\ref{PolynomialSchur}}}{ = } \;\; K_0(\Pol_d(R)) \to \bZ[s_1,s_2,\ldots]_d
\]
given by sending a polynomial functor $F$ to the weights of the action of $\bG_m^n$ on $F(R^n)$.

\begin{thm}
\label{thm:K0-over-field}
For every field $K$ the map $K_0(\cM(K))\to \bZ[s_1,s_2,\ldots]_d$ is an isomorphism.
\end{thm}

\begin{proof} See \cite[2.2 \& 3.5, esp.~Remark (ii)]{Gr}. The book \cite{Gr} assumes the field $K$ to be infinite, but this assumption is only used in relating modules over $\Gamma^d\Mat(n,K)$ to representations of the monoid $\Mat(n,K)$, as opposed to representations of the \emph{monoid scheme} $\Mat_{n,K}$, which would also work over a finite field $K$. See also \cite[II.A]{Jantzen}.
\end{proof}

We will deduce from the cases $K=\bQ$ and $K=\bF_p$ in this theorem that the map
\[
	K_0(\Pol_d(\bZ)) \to \bZ[s_1,s_2,\ldots]_d
\]
is an isomorphism. The proof is essentially identical to Serre's proof that
\[
	K_0(\GL_{n,\bZ}) \to K_0(\GL_{n,\bQ})
\]
is an isomorphism, see \cite{Serre}.

\begin{lem}[Projective resolutions]\label{lemma:projective-resolution}
The canonical map
\[
	K_0(\cM_\mathrm{p}(\bZ))\to K_0(\cM(\bZ))
\]
is an isomorphism.
\end{lem}

\begin{proof} (Compare \cite[\S 2.2 \& \S 2.3]{Serre}.) Let $M$ be a finitely generated module over $\Gamma^d\Mat(n,\bZ)$. If $M$ can be generated by $m$ elements, we obtain a presentation
\[
	0 \to P_0 \to P_1 \to M \to 0
\]
with $P_1=(\Gamma^d \Mat(n,\bZ))^m$. Both $P_1$ and $P_0$ are torsion-free, hence projective as $\bZ$-modules.
The same argument as \cite[Prop.~4]{Serre} shows that $[P_0]-[P_1] \in K_0(\cM_\mathrm{p}(\bZ))$ is independent of the choice of presentation, and that $M \mapsto [P_0]-[P_1]$ defines a two-sided inverse to the map of the proposition.
\end{proof}

\begin{lem}[Localization sequence]\label{lemma:localization-sequence}
The obvious sequence
\[
	\mathop{\oplus}_{ \ell \;\mathrm{ prime}}K_0(\cM(\bF_\ell)) \to K_0(\cM(\bZ)) \to K_0(\cM(\bQ)) \to 0
\]
is exact.
\end{lem}

\begin{proof}
The argument is identical to \cite[Th\'eor\`eme 1]{Serre}. The main point is to verify that
every $\Gamma^d\Mat(n,\bQ)$-module $V$ of finite $\bQ$-dimension contains
a $\Gamma^d\Mat(n,\bZ)$-submodule $\Lambda$ with $\bQ\otimes_\bZ \Lambda = V$. To construct such $\Lambda$,
take an arbitrary sub-$\bZ$-module $\Lambda_0$ with $\bQ\otimes_\bZ \Lambda_0 = V$, and take
$\Lambda := \Gamma^d\Mat(n,\bZ) \Lambda_0$.
\end{proof}

\begin{lem}[Decomposition maps] For every prime $\ell$ there is a unique homomorphism
$d_\ell$ making the triangle
\[
\begin{tikzcd}
	K_0(\cM_p(\bZ)) \arrow{r} \arrow{d} & K_0(\cM(\bQ)) \arrow{ld}{d_\ell}\\
	K_0(\cM(\bF_\ell))
\end{tikzcd}
\]
commute.
\end{lem}

\begin{proof}The argument is identical to \cite[Thm~2]{Serre}.
\end{proof}

\begin{lem}
\label{lemma:composite-zero}
The composition
\[
	K_0(\cM(\bQ)) \overset{d_\ell}{\longto} K_0(\cM(\bF_\ell))
	\longto K_0(\cM(\bZ))
\]
is the zero map.
\end{lem}

\begin{proof}See \cite[Lemme 4]{Serre}. Similarly to there, one uses that for every $\bZ$-torsion-free $\Gamma^d\Mat(n,\bZ)$-module $\Lambda$ the map $\Lambda \to \ell \Lambda$, $x \mapsto \ell x$, is an isomorphism of  $\Gamma^d\Mat(n,\bZ)$-modules.
\end{proof}

\begin{prop}
\label{prop:dec-iso}
For every prime $\ell$, the map $d_\ell$ is an isomorphism.
\end{prop}

\begin{proof}
By Theorem~\ref{thm:K0-over-field} the functors
\[
	\Lambda^{d_1} \otimes \Lambda^{d_2} \otimes \cdots \otimes \Lambda^{d_m}
\]
with $\sum d_i=d$ define a basis of $K_0(\cM(\bQ))$ and of $K_0(\cM(\bF_\ell))$. Since the map $d_\ell$ preserves this basis, it is an isomorphism.
\end{proof}

\begin{cor}\label{Grothendieck group integers equals Grothendieck group rationals}The canonical map
\[
	K_0(\cM(\bZ)) \to K_0(\cM(\bQ))
\]
is an isomorphism.
\end{cor}

\begin{proof}
By Proposition~\ref{prop:dec-iso} and Lemma~\ref{lemma:composite-zero}
the maps
\[
	K_0(\cM(\bF_\ell)) \to K_0(\cM(\bZ))
\]
are the zero maps. But then the localization sequence of Lemma~\ref{lemma:localization-sequence} shows that the map $K_0(\cM(\bZ)) \to K_0(\cM(\bQ))$ is an isomorphism.
\end{proof}

\begin{proof}[Proof of Theorem~\ref{Computing Grothendieck group of PolFun}] The degree $d$ part of the homomorphism
\[\bZ[s_1, s_2, \ldots ] \rightarrow K_0(\mathrm{Pol}_{< \infty}(\bZ)), \quad s_i \mapsto [\Lambda^i], \]
is obviously inverse to the composition of the isomorphisms
\[K_0(\mathrm{Pol}_d(\bZ)) \cong K_0(\cM_\mathrm{p}(\bZ)) \cong K_0(\cM(\bZ)) \cong K_0(\cM(\bQ)) \cong \bZ[s_1, s_2, \ldots ]_d\]
given by Theorem~\ref{PolynomialSchur}, Lemma~\ref{lemma:projective-resolution}, Corollary~\ref{Grothendieck group integers equals Grothendieck group rationals} and Theorem~\ref{thm:K0-over-field}, respectively, and is hence bijective and compatible with exterior power operations.
\end{proof}

\begin{rem}
The category $\Pol_{d}(R)$ is equivalent with the category of weight $d$ representations of the monoid $\Mat_{n,R}$, which forms a full subcategory of the category of representations of $\GL_{n,R}$. Rather than translating Serre's argument from the $\GL_n$ to the $\Mat_n$ context, one could also  deduce our result from Serre's. However, some care has to be taken because the right adjoint to the inclusion, mapping a $\GL_{n,\bZ}$-representation $V$ to the largest sub-representation that extends to $\Mat_{n,\bZ}$, is not exact, see \cite[II.A]{Jantzen}.
\end{rem}

\subsection{Proof of the final \texorpdfstring{$\lambda$}{lambda}-ring axiom}
\label{Proof of final axiom}

In this subsection, we derive from Corollary~\ref{PolynomialFunctorsFormLambdaRing} that, for every quasi-compact scheme $X$, the pre-$\lambda$-ring $K_*(X)$ satisfies the final $\lambda$-ring axiom. Together with Theorem~\ref{lambda axiom 2}, this implies that $K_*(X)$ is a $\lambda$-ring and finishes the proof of this paper's main result.

\begin{thm}
\label{lambda axiom 3}
The ring $K_*(X)$ equipped with the exterior power operations defined in
section~\ref{Lambda axiom 4} satisfies axiom~(3) of a $\lambda$-ring.
\end{thm}

\begin{proof}
Let $r,s \ge 1$ and $n\ge 0$. For every $x \in K_n(X)$ we want to show that the identity
\begin{equation}\label{Axiom5}
\lambda^r(\lambda^s(x)) = P_{r,s}(\lambda^1(x), \ldots, \lambda^{rs}(x))
\end{equation}
holds in $K_n(X)$. We recall that, if $n \ge 1$, all products occurring on the right-hand side of (\ref{Axiom5}) are trivial (and hence that the right-hand side of~(\ref{Axiom5}) happens to be just a multiple of $\lambda^{rs}(x)$). We will show the stronger statement that the identity~(\ref{Axiom5}) in fact holds in $K_0\left((B_\mathrm{b}^\mathrm{q})^n \mc{P}(X)\right)$ for all $x \in K_0\left((B^\mathrm{q}_\mathrm{b})^n \mc{P}(X)\right)$. Now the products occurring on the right-hand side of~(\ref{Axiom5}) are induced by the simplicial tensor product introduced in section~\ref{Simplicial tensor products}; these products become trivial in $K_n(X)$ by Proposition~\ref{tensor product vanishes}. By a standard argument (see the proof of Theorem~\ref{lambda axiom 2}) we may assume that $x$ is the class of an object $P_\cdot$ of $(B_\mathrm{b}^\mathrm{q})^n \mc{P}(X)$. One easily checks, for instance using the Gabriel--Quillen embedding theorem \cite[A.7.1, A.7.16]{Thomason}, that for every exact category $\mathcal{P}$ and any skeletally small category $\mathcal I$, the category of functors from $\mathcal I$ to $\mathcal P$ is again an exact category in the obvious way. In particular, the category $\mathrm{End}\left((B_\mathrm{b}^\mathrm{q})^n \mc{P}(X)\right)$ of endo-functors of $(B_\mathrm{b}^\mathrm{q})^n \mc{P}(X)$ is an exact category. Furthermore it carries a tensor product and exterior power operations (given by $F \mapsto \Lambda^d_n \circ F$). Via the homomorphism $K_0\left( \mathrm{End}\left((B_\mathrm{b}^\mathrm{q})^n \mc{P}(X)\right)\right) \rightarrow K_0 \left((B_\mathrm{b}^\mathrm{q})^n \mc{P}(X)\right)$ given by $F \mapsto F(P_\cdot)$, the desired identity now follows from the even stronger identity
\begin{equation}\label{Axiom5strong}
[\Lambda^r_n \circ \Lambda^s_n] = P_{r,s}([\Lambda_n^1], \ldots, [\Lambda_n^{rs}])
\end{equation}
in $K_0\left( \mathrm{End}\left((B_\mathrm{b}^\mathrm{q})^n\mc{P}(X) \right)\right)$, which we now prove. We remember that the identity~(\ref{Axiom5strong}) (with the subscripts~$n$ omitted) holds in the Grothendieck group $K_0 \left(\mathrm{Pol}_{< \infty}(\bZ)\right)$ by Corollary~\ref{PolynomialFunctorsFormLambdaRing}. Then it also holds in $K_0\left( \mathrm{Pol}_{< \infty}^0(\bZ)\right)$ where $\mathrm{Pol}_{< \infty}^0(\bZ)$ denotes the full subcategory of $\mathrm{Pol}_{< \infty}(\bZ)$ consisting of functors $F$ satisfying $F(0)=0$; this follows from the fact that the canonical inclusion $\mathrm{Pol}^0_{< \infty}(\bZ) \hookrightarrow \mathrm{Pol}_{< \infty}(\bZ)$ is split by $F \mapsto (V \mapsto \mathrm{ker}(F(V) \rightarrow F(0)))$. The identity~(\ref{Axiom5strong}) therefore follows from Corollary~\ref{PolynomialFunctorsFormLambdaRing} once we have shown that we have a pre-$\lambda$-ring homomorphism
\begin{equation}\label{homom}
K_0 \left(\mathrm{Pol}^0_{< \infty}(\bZ)\right) \rightarrow K_0 \left(\mathrm{End}\left((B^\mathrm{q}_\mathrm{b})^n \mc{P}(X) \right)\right)
\end{equation}
that sends the class of the identity functor to the class of the identity functor.
By base change (see subsection~\ref{PolynomialFunctor}), every functor in $\mathrm{Pol}(\bZ)$ induces a functor in $\mathrm{Pol}(U)$ for every open subset $U$ of $X$ and this construction is compatible with restriction with respect to any inclusion of open subsets of~$X$. The constructions of sections~\ref{Operations on acyclic complexes}, \ref{Operations on binary multicomplexes} and \ref{proof of main theorem} therefore inductively induce a functor
\[\mathrm{Pol}^0_{< \infty}(\bZ) \rightarrow \mathrm{End}\left((B_\mathrm{b}^\mathrm{q})^n \mc{P}(X)\right);\]
this functor is exact and compatible with tensor products and exterior power operations as one easily verifies by induction on $n$. Thus it induces the desired homomorphism~(\ref{homom}) and the proof of Theorem~\ref{lambda axiom 3} is complete.
\end{proof}

\begin{rem}\label{remark} We have seen in the previous proof that the $\lambda$-ring axiom~(3)
already holds in $K_0((B_\mathrm{b}^\mathrm{q})^n \proj (X))$,
i.e.\@, before dividing out the subgroup generated by classes of diagonal multicomplexes.
The same holds true for the $\lambda$-ring axiom (2).
This can be shown similarly by using Corollary~\ref{PolynomialFunctorsFormLambdaRing} or by using the characteristic-free Cauchy decomposition as constructed in \cite{ABW}.
Whereas Corollary~\ref{PolynomialFunctorsFormLambdaRing} only proves the existence of short exact sequences, Akin, Buchsbaum and Weyman {\em explicitly} construct short exact sequences in \cite{ABW} that prove axiom~(2) of a $\lambda$-ring. The problem of explicitly describing  short exact sequences of polynomial functors that prove axiom~(3) seems however to be even harder than the famous and related plethysm problem in representation theory. Such explicit short exact sequences for the plethysm $\Lambda^2 \circ \Lambda^2$ can be found on page 175 of the paper \cite{AB85} by Akin and Buchsbaum. Although there also exist solutions of the classical plethysm problem for $\Lambda^r \circ \Lambda^2$ and $\Lambda^2 \circ \Lambda^s$, we are not aware of any corresponding characteristic-free short exact sequences.
\end{rem}

\appendix

\section{Proof of Lemma~\ref{iterate lemma}}
\label{acyclic app}

In this appendix we prove Lemma~\ref{iterate lemma}, which states that:
\begin{enumerate}
 \item If $\mc{P}$ is an idempotent complete exact category, then so are $C^n\mc{P}$ and $(C^{\rm{q}})^n\mc{P}$.
 \item If $\mc{P}$ is a split exact category, then so is $(C^{\rm{q}})^n\mc{P}$.
\end{enumerate}
Note that to prove each of these statements it is enough to prove the case $n = 1$.

\begin{proof}[Proof of Lemma~\ref{iterate lemma} (1)]
 Let $e \colon P_{\cd} \rightarrow P_{\cd}$ be an idempotent map of chain complexes.
 Then each map $e_n : P_n \rightarrow P_n$ is an idempotent of $\mc{P}$ and so has a kernel $\operatorname{ker}(e_n)$
 which is an object of $\mc{P}$.
 By the universal property of kernels, the chain map on $P_{\cd}$ induces a map
 $\operatorname{ker}(e_n) \rightarrow \operatorname{ker}(e_{n-1})$ for each $n$, and these assemble to form a chain complex of
 kernels. Thus every idempotent in $C\mc{P}$ has a kernel in $C\mc{P}$, so $C\mc{P}$ is idempotent complete.
 To show that $C^{\rm{q}}\mc{P}$ is idempotent complete as well, we must show
 that this kernel chain complex is acyclic in $\mc{P}$ if $P_{\cd}$ is.
 To do this, it suffices to consider the case when the complex is a short exact sequence;
 the general case then follows because $\mc{P}$ supports long exact sequences.\\
 If $P_{\cd} = (0 \rightarrow P_2 \rightarrow P_1 \rightarrow P_0 \rightarrow 0)$ and if
 $e \colon P_{\cd} \rightarrow P_{\cd}$ is an idempotent chain map, then $P_{\cd}$ is isomorphic to
 a sequence of the form
 \[
  0 \longrightarrow \ke (e_2) \oplus \im (e_2) \longrightarrow \ke (e_1) \oplus \im (e_1) \longrightarrow
  \ke (e_0) \oplus \im (e_0) \longrightarrow 0.
 \]
 Furthermore, as the morphisms in this short exact sequence commute with the idempotents
 $\left(\begin{smallmatrix} 0 & 0 \\ 0 & 1 \end{smallmatrix}\right)$, they split as direct sums of induced morphisms.
 Hence the sequence
 \[
  0 \longrightarrow \ke (e_2) \longrightarrow \ke (e_1) \longrightarrow
  \ke (e_0) \longrightarrow 0
 \]
 is exact as well.
\end{proof}

\begin{proof}[Proof of Lemma~\ref{iterate lemma} (2)]
  We wish to show that every admissible monomorphism $i \colon P_{\cd} \rightarrowtail Q_{\cd}$ in $C^{\rm{q}}\mc{P}$ is split;
 that is, that there exists a chain map $s \colon Q_{\cd} \rightarrow P_{\cd}$ such that each
 $s_ni_n \colon P_n \rightarrow Q_n$ is the identity.
 Let us restrict to the case in which $P_{\cd}$ and $Q_{\cd}$ are short exact sequences of $\mc{P}$.
 Consider the following diagram in a split exact category,
 \[ \xymatrix{
  P'   \ar@{ >->}[r]^{j_P} \ar@{ >->}[d]^{i'} &
  P    \ar@{->>}[r]^{q_P}  \ar@{ >->}[d]^i &
  P''  \ar@{ >->}[d]^{i''} \\
  Q'   \ar@{ >->}[r]^{j_Q} &
  Q    \ar@{->>}[r]^{q_Q}   &
  Q''  \\
  }
 \]
 and fix a splitting $s''$ for $i''$. We claim that there exist splittings $s'$ and $s$ of $i'$ and $i$ such that
 the resulting $s_{\cd} \colon Q_{\cd} \rightarrow P_{\cd}$ is a chain map (and hence a splitting of $i_{\cd}$).
 The general case follows from this claim. Indeed, since acyclic complexes are spliced together from short exact sequences,
 we construct a splitting for a monomorphism of acyclic complexes
 $i \colon P_{\cd} \rightarrowtail Q_{\cd}$ by splitting each monomorphism of short exact sequences separately.
 The part of the claim concerning a fixed splitting $s''$ allows us to choose these splittings of short exact sequences
 in a compatible manner (beginning in degree 0). So it is enough to prove the claim. \\
 We choose compatible splittings $h_P$ and $t_P$ of $j_P$ and $q_P$ respectively, i.e.\@,
 $j_Ph_P + t_Pq_P = 1$. We also choose a splitting $s_0$ for $i$ and now
 set $s = j_Ph_Ps_0 + t_Ps''q_Q$.
 Then we compute:
 \begin{enumerate}
  \item $si = j_Ph_Ps_0i + t_Ps''q_Qi = j_Ph_P + t_Ps''i''q_P = j_Ph_P + t_Pq_P = 1$,
  \item $q_Ps = q_Pj_Ph_Ps_0 + q_Pt_Ps''q_Q = s''q_Q$,
 \end{enumerate}
 so $s$ is a splitting for $i$ and $s, s''$ commute with $q_P, q_Q$. We therefore get an induced map of kernels
 $s' \colon Q' \rightarrow P'$ satisfying $j_Ps' = sj_Q$. Moreover
 $j_Ps'i' = sj_Qi' = sij_P = j_P$, and $j_P$ is monic, so $s'i' =1$.
\end{proof}

\bibliographystyle{amsalpha}

\end{document}